\documentclass[sn-mathphys-ay]{sn-jnl}

\usepackage{graphicx}%
\usepackage{multirow}%
\usepackage{amsmath,amssymb,amsfonts}%
\usepackage{amsthm}%
\usepackage{mathrsfs}%
\usepackage[title]{appendix}%
\usepackage{xcolor}%
\usepackage{textcomp}%
\usepackage{manyfoot}%
\usepackage{booktabs}%
\usepackage{algorithm}%
\usepackage{algorithmicx}%
\usepackage{algpseudocode}%
\usepackage{listings}%
\usepackage{enumerate}%

\theoremstyle{thmstyleone}%
\newtheorem{theorem}{Theorem}
\newtheorem{lemma}{Lemma}
\newtheorem{corollary}{Corollary}

\theoremstyle{thmstyletwo}%

\theoremstyle{thmstylethree}%

\begin{document}

\title[Consistency of Bayes factors for linear models]{Consistency of Bayes factors for linear models}


\author*[1]{\fnm{El{\'\i}as} \sur{Moreno}}\email{emoreno@ugr.es}

\author[2]{\fnm{Juan J.} \sur{Serrano-P{\'e}rez}}\email{jjserra@ugr.es}

\author[2]{\fnm{Francisco} \sur{Torres-Ruiz}}\email{fdeasis@ugr.es}

\affil*[1]{\orgname{Royal Academy of Sciences}, \orgaddress{\country{Spain}}}


\affil[2]{\orgdiv{Department of Statistic}, \orgname{University of Granada}, \orgaddress{
\country{Spain}}}



\abstract{The quality of a Bayes factor crucially depends on the number of regressors, the sample
size and the prior on the regression parameters, and hence it has to be established in a
case-by-case basis.

In this paper we analyze the consistency of a wide class of Bayes factors when the number of
potential regressors grows as the sample size grows.

We have found that when the number of regressors is finite some classes of priors yield
inconsistency, and\ when the potential number of regressors grows at the same rate than the sample
size different priors yield different degree of inconsistency.

For moderate sample sizes, we evaluate the Bayes factors by comparing the posterior model
probability. This gives valuable information to discriminate between the priors for the model
parameters commonly used for variable selection.}




\maketitle

\section{Introduction}\label{sec1}

Bayesian variable selection in linear model is based on the posterior probability of a given set
of candidate models, and, for convenience, these posterior probabilities are defined in terms of
Bayes factors that compare nested models: a generic normal regression model $M_{p}$ with $p\geq 1$
regressors against the intercept only model $M_{0}$ (encompassing from below).

The Bayes factor introduced by \cite{Jeffreys1961}, the main statistical tool for model selection
as it contains all the data information for model comparison, crucially depends on the number of
regressors and the prior distribution of the regression parameters and, to a lesser extent, on the
prior of the variance error. A recent review of priors for model selection is given in
\cite{Consonni2018}.

In this paper we deal with a set of objective and subjective Bayes factors, the former are defined
by objective intrinsic priors and the others by subjective mixtures of the Zellner's g-prior. All
of them have in common their dependency on the data through the ratio of the sum of squared of the
residuals of the models, the sufficient statistic, the number of regressors $p$ and the sample
size $n$. They certainly differ in the way the sufficient statistic is corrected by $(p,n)$.

The intrinsic priors were introduced by \cite{Berger1996} to justify the arithmetic intrinsic
Bayes factor, an empirical tool for model selection based on the notion of partial Bayes factor
\citep{Leamer1983}. These priors
were studied as priors for model selection in \cite{Moreno1997} and \cite%
{Moreno1998}, and they have been applied to variable selection in regression by \cite{Moreno2003},
\cite{Moreno2005a}, \cite{Moreno2005b,Moreno2008}, \cite{Casella2006}, \cite{Giron2006},
\cite{Giron2010}, \cite{Torres2011}, \cite{Kang2021}, among others.

On the other hand, the $g$-prior was introduced by \cite{Zellner1980} and since then a
considerable number of mixtures of the $g$-prior for model selection have
been proposed including those by \cite{Zellner1984}, \cite{Fernandez2001}, \cite%
{Cui2008}, \cite{Liang2008}, \cite{Maruyama2011}, \cite{Bayarri2012} and \cite{Giron2021}.

We analyze here the asymptotic of six Bayes factors for comparing the generic model $M_{p}$
against $M_{0}$. We remark that the inconsistency of a Bayes factor implies the inconsistency of
the posterior model probability in the set of candidate models \citep{Moreno2015}. Thus, an
inconsistent Bayes factor should be rejected for variable selection.

When the number of regressors $p$ is either finite or grows at rate $p=O(n^{b})$, $0<b<1$, the
Bayes factors that use either intrinsic priors or conventional mixtures of $g$-priors
\citep{Zellner1984, Fernandez2001} are consistent having the same rate of convergence than that of
the BIC approximation. However, Bayes factors for others mixtures of $g$-priors are inconsistent
when sampling from $M_{0}$. Further when $p$ grows at the same rate than the sample size $n$,
$p=O(n)$, as occurs in clustering or the analysis of variance, everyone is inconsistent when
sampling from specific sets of models. These sets are given and compared each other.

For finite sample size $n$ and for each Bayes factor the posterior probability of $M_0$ in the set
$\{M_{0}, M_{p}\}$ as a function of the sufficient statistic provide relevant information on the
entertained Bayes factors.

The rest of the paper is organized as follows. In Section \ref{mod} the model notation is
presented. In Section \ref{priorsBF} the set of priors for the parameters of $M_{0}$ and $M_{p}$
and the Bayes factors are given. Section \ref{asymp} presents the asymptotic analysis and Section
\ref{pprob} the posterior model probability for finite sample size. Section \ref{conc} contains
concluding remarks.

\section{The models}\label{mod}

Let $Y$ represent the response random variable and $\mathbf{x}=(x_{1},\ldots ,x_{p})$ a vector of
explanatory deterministic regressors related through the normal linear model
$$
Y=\beta_{0}+\beta_{1}x_{1}+\ldots +\beta_{p}x_{p}+\varepsilon_{p},
$$
where $\boldsymbol{\beta}_{p+1}=(\beta_{0},\beta_{1},\ldots ,\beta_{p})^{\prime}$ is the vector of
regression coefficients, $\varepsilon_{p}\sim N(0,\sigma_{p}^{2})$ and $\sigma_{p}^{2}$ is the
variance error. The sampling distribution of $Y$ is assumed to be either the normal distribution
with $p$ regressors $N(y\mid \mathbf{x} \boldsymbol{\beta}_{p+1},\sigma _{p}^{2})$ or the
intercept only model $N(y\mid \alpha_{0},\sigma_{0}^{2})$. \smallskip

Let $(\mathbf{y},\mathbf{X})$ be the data set, where $\mathbf{y}$ is a vector of $n$ independent
observations of $Y$, and $\mathbf{X}$ a $n\times (p+1)$ design matrix of full rank. The likelihood
of $(\boldsymbol{\beta}_{p+1},\sigma_{p})$ is given by $N_{n}(\mathbf{y\mid
X}\boldsymbol{\beta}_{p+1},\sigma_{p}^{2}\mathbf{I}_{n})$, and for a given prior distribution $\pi
(\boldsymbol{\beta}_{p+1},\sigma_{p})$ the Bayesian model is denoted as
$$
M_{p}:\{N_{n}(\mathbf{y\mid X}\boldsymbol{\beta}_{p+1},\sigma_{p}^{2} \mathbf{I}_{n}),\pi
(\boldsymbol{\beta}_{p+1},\sigma_{p})\}.
$$
Likewise, the likelihood of the parameter of the intercept only model $(\alpha_{0},\sigma_{0})$ is
$N_{n}(\mathbf{y}\mid \alpha_{0}\mathbf{1}_{n},\sigma _{0}^{2}\mathbf{I}_{n})$, and for a given
prior distribution $\pi (\alpha_{0},\sigma_{0})$ the Bayesian model is
$$
M_{0}:\{N_{n}(\mathbf{y}\mid \alpha_{0}\mathbf{1}_{n},\sigma_{0}^{2} \mathbf{I}_{n}),\pi
(\alpha_{0},\sigma_{0})\}.
$$

\noindent Sometimes the models are simplified assuming that the variances of $M_{0}$ and $M_{p}$
are equal.

\section{Priors for the model parameters and Bayes factors}\label{priorsBF}

\subsection{Intrinsic priors and Bayes factor for heteroscedastic models}

\cite{Moreno2003} and \cite{Giron2006} used the methodology in \cite{Berger1996} and
\cite{Moreno1998} for constructing the intrinsic priors for the parameters of the normal linear
models from the improper reference priors \citep{Berger1992}.

The intrinsic prior for the parameters $(\boldsymbol{\beta}_{p+1},\sigma_{p})$ of $M_{p}$,
conditional on $(\alpha_{0},\sigma_{0})$, turns out to be \citep{Moreno2003}
$$
\pi^{IP}(\boldsymbol{\beta}_{p+1},\sigma_{p}\mid \alpha_{0},\sigma_{0}) = N\!\left(
\boldsymbol{\beta}_{p+1} \mid \boldsymbol{\tilde{\alpha}}_{0},
\tfrac{n(\sigma_{p}^{2}+\sigma_{0}^{2})}{p+2}(\mathbf{X}^{\prime} \mathbf{X})^{-1} \right)
\smallskip \,H\!C^{+}(\sigma_{p}\,|\,0,\sigma_{0}),
$$
where $\boldsymbol{\tilde{\alpha}}_{0}=(\alpha_{0},\mathbf{0}_{p}^{\prime})^{\prime}$, and
$H\!C^{+}(\sigma_{p}\mid 0,\sigma_{0})$ is the half Cauchy distribution on the positive part of
the real line with location at $0$ and scale $\sigma_{0}$. Thus, the unconditional intrinsic prior
for $(\boldsymbol{\beta}_{p+1},\sigma_{p})$ is given as
$$
\pi^{IP}(\boldsymbol{\beta}_{p+1},\sigma_{p})= c \ {\int\nolimits_{0}^{\infty} \negthinspace
\negthinspace \int\nolimits_{-\infty}^{\infty} \negthinspace N\!\left( \boldsymbol{\beta}_{p+1}
\mid
\boldsymbol{\tilde{\alpha}}_{0},\tfrac{n(\sigma_{p}^{2}+\sigma_{0}^{2})}{p+2}(\mathbf{X}^{\prime}
\mathbf{X})^{-1}\right)}\,H\!C^{+}(\sigma_{p}\,|\,0,\sigma_{0})\,\tfrac{1}{\sigma_{0}} \,
d\alpha_{0} \, d\sigma_{0}.
$$

This is an improper prior although the Bayes factor of $M_{p}$ against $M_{0}$ for
$(\pi^{N}(\alpha_{0},\sigma_{0})$, $\pi^{IP}(\boldsymbol{\beta}_{p+1},\sigma_{p}))$ is
well-defined as the arbitrary constant $c$ that appears in $\pi^N(\alpha_0,\sigma_0)$ cancels out
in the ratio. \smallskip

For a sample $\mathbf{(y,X)}$ from either model $M_{p}$ or $M_{0}$, the Bayes factor for the
intrinsic priors $\big(\pi^{N}(\alpha_{0},\sigma_{0}),
\pi^{IP}(\boldsymbol{\beta}_{p+1},\sigma_{p}) \big)$ is given by
\begin{equation}
B_{p0}^{IP}(\mathbf{y},\mathbf{X})=\frac{(p+2)^{p/2}}{\pi/2} \int_{0}^{\pi /2}\!\frac{\sin
^{p}\!\raisebox{2pt}{\fontsize{9pt}{11pt} \selectfont$\varphi$}\,[n+(p+2)\sin
^{2}\raisebox{2pt}{\fontsize{9pt}{11pt}\selectfont$\varphi$}]^{(n-p-1)/2}}{[(p+2)\sin
^{2}\!\raisebox{2pt}{\fontsize{9pt}{11pt}\selectfont$\varphi$}+n\mathcal{B}_{p0}]^{(n-1)/2}}\,d
\raisebox{2pt}{\fontsize{9pt}{11pt}\selectfont$\varphi$},  \label{IPBF}
\end{equation}
where
$$
\mathcal{B}_{p0}=\frac{\mathbf{y}^{\prime }(\mathbf{I}_{n}-\mathbf{H})
\mathbf{y}}{\mathbf{y}^{\prime }(\mathbf{I}_{n}-\frac{1}{n}\mathbf{1}_{n} \mathbf{1}_{n}^{\prime
})\mathbf{y}},
$$
and $\mathbf{H}=\mathbf{X}(\mathbf{X}^{\prime }\mathbf{X})^{-1}\mathbf{X}^{\prime}$. The integral
in (\ref{IPBF}) does not have a closed form expression and needs numerical integration.

\subsection{Intrinsic prior and Bayes factor for homoscedastic models}

If the variance $\sigma^{2}$ is assumed to be a common parameter of the models, the intrinsic
prior for $(\boldsymbol{\beta}_{p+1},\sigma)$ can be shown to be
\[
\pi^{IPH}(\boldsymbol{\beta}_{p+1}|\alpha _{0},\sigma ) = N_{p+1} \left( \boldsymbol{\beta}_{p+1}|
\boldsymbol{\tilde{\alpha}}_{0}, \frac{2n}{p+1} \sigma^{2} (\mathbf{X}_{p+1}^{\prime}
\mathbf{X}_{p+1})^{-1} \right),
\]
and the Bayes factor of $M_{p}$ against $M_{0}$ for the priors $\big( \pi^{N}(\alpha_{0},\sigma),
\pi^{IPH}(\boldsymbol{\beta}_{p+1},\sigma) \big)$ turns out to be
$$
B_{p0}^{IPH}(\mathbf{y},\mathbf{X})=\frac{\left( 1+\dfrac{2n}{p+1}\right)^{(n-p-1)/2}}{\left(
1+\dfrac{2n}{p+1}\,\mathcal{B}_{p0}\right)^{(n-1)/2}}.
$$

\subsection{Mixtures of the Zellner's $g$-prior and Bayes factors}

The stream of research that compares $M_{p}$ and $M_{0}$ using mixtures of $g$-prior simplifies
the sampling models assuming that the intercept $\beta_{0}$ and the variance error $\sigma^{2}$
are common parameters. The prior for $(\beta_{0},\sigma)$ is assumed to be the improper reference
prior
$$
\pi ^{N}(\beta_{0},\sigma)=\frac{c}{\sigma},
$$
where $c$ is an arbitrary positive constant, and the prior distribution for the rest of the
regression coefficients $\boldsymbol{\beta}_{p}=(\beta_{1}, \ldots,\beta_{p})^{\prime }$ of
$M_{p}$\smallskip is assumed to be the Zellner's $g$-prior $N(\boldsymbol{\beta}_{p} \mid
\mathbf{0}_{p}, g\sigma^{2}(\mathbf{X}_{p}^{\prime} \mathbf{X}_{p})^{-1})$, where $g$ is a unknown
positive hyperparameter and $\mathbf{X}_{p}$ is the matrix obtained from the original $\mathbf{X}$
by suppressing their first column.

\subsubsection{The Zellner-Siow's mixture of $g$-prior}

\cite{Zellner1980} implicitly assumed for the hyperparameter $g$ the Inverse Gamma distribution
$$
\pi^{ZS}(g \,|\, n)=\frac{(n/2)^{1/2}}{\Gamma (1/2)}\,g^{\raisebox{1.5pt}{$\scriptstyle
-3/2$}}exp\left( -n/(2g)\right) ,
$$
and hence the prior for the parameters $(\beta_{0},\boldsymbol{\beta}_{p},\sigma)$ of model
$M_{p}$ is the mixture
$$
\pi^{ZS}(\beta_{0},\boldsymbol{\beta}_{p},\sigma )=\displaystyle{\frac{c}{
\sigma}\,\int_{0}^{\infty }N(\boldsymbol{\beta }_{p}\mid \mathbf{0}_{p},g\sigma
^{2}(\mathbf{X}_{p}^{\prime}\mathbf{X}_{p})^{-1})}\,\pi ^{ZS}(g \,|\, n)\,dg.
$$
This is an improper prior whose arbitrary constant $c$ is that of
$\pi^{N}(\alpha_{0},\sigma_{0})$, and therefore the Bayes factor of $M_{p}$ against $M_{0}$ for
$\big( \pi^{N}(\beta_{0},\sigma), \pi^{ZS}(\beta_{0}, \boldsymbol{\beta}_{p},\sigma) \big)$ is
well-defined and it turns out to be
\begin{equation}
B_{p0}^{ZS}(\mathbf{y},\mathbf{X})=\frac{(n/2)^{1/2}}{\Gamma (1/2)} \int_{0}^{\infty}
\frac{(1+g)^{(n-p-1)/2}}{(1+g\mathcal{B}_{p0})^{(n-1)/2}}\,g^{\raisebox{1.5pt}{$\scriptstyle
-3/2$}}\,exp\left( -n/(2g)\right) dg. \label{ZSBF}
\end{equation}

\noindent The integral in (\ref{ZSBF}) does not have a closed form expression.

\subsubsection{A degenerate mixtures of $g$-prior} \smallskip

Several alternative priors to the distribution $\pi^{ZS}(g)$ have been considered. \smallskip The
simplest one is the degenerate prior $\pi^{FS}(g \,|\, n)=1_{(n)}(g)$ \smallskip suggested by
\cite{Zellner1984} and further studied by \cite{Fernandez2001}. For $\pi^{FS}(g \,|\, n)$ the
prior for the parameter $(\beta_{0},\boldsymbol{\beta}_{p},\sigma)$ of model $M_{p}$ is given by
$$
\pi^{FS}(\beta_{0},\boldsymbol{\beta}_{p},\sigma)=\frac{c}{\sigma} \, N(\boldsymbol{\beta}_{p}\mid
\mathbf{0}_{p},n\sigma^{2}(\mathbf{X}_{p}^{\prime }\mathbf{X}_{p})^{-1}).
$$
The Bayes factor of $M_{p}$ against $M_{0}$ for $\big(\pi^{N}(\beta_{0},\sigma),
\pi^{FS}(\beta_{0},\boldsymbol{\beta}_{p},\sigma)\big)$ is well-defined and it turns out to be
$$
B_{p0}^{FS}(\mathbf{y},\mathbf{X})=\frac{(1+n)^{(n-p-1)/2}}{(1+n\mathcal{B}_{p0})^{(n-1)/2}}.
$$

\subsubsection{The prior $\pi^{L}(g \,|\, n)$}

A more sophisticated prior for $g$ than the degenerate one was introduced by \cite{Liang2008} and
is given by
$$
\pi ^{L}(g)=\frac{1}{2}(1+g)^{-3/2}.
$$
However, for this prior the Bayes factor
$$
B_{p0}^{\pi^{L}(g)}(\mathbf{y},\mathbf{X})=\frac{1}{2}\int_{0}^{\infty}
\frac{(1+g)^{(n-p-1)/2}}{(1+g\mathcal{B}_{p0})^{(n-1)/2}}\,\left( 1+g \right)^{-3/2}dg
$$
is inconsistent, its limit in probability is infinity when sampling from the model $M_{0}$. This
drawback of $\pi(g)$ motivated to \cite{Liang2008} to introduce the prior
$$
\pi^{L}(g \,|\, n)=\frac{1}{2n}(1+g/n)^{-3/2}
$$
that depends on $n$. Then, for $\pi^{L}(g \,|\, n)$ the prior for $(\beta_{0},
\boldsymbol{\beta}_{p},\sigma)$ is given by
$$
\pi^{L}(\beta_{0},\boldsymbol{\beta}_{p},\sigma)=\displaystyle{\frac{c}{2n\sigma
}\int_{0}^{\infty}N_{p}(\boldsymbol{\beta}_{p}\mid \mathbf{0}_{p},g\sigma
^{2}(\mathbf{X}_{p}^{\prime}\mathbf{X}_{p})^{-1})}\ \left( 1+g/n \right)^{\!-3/2}dg.
$$
For $\big(\pi^{N}(\beta_{0},\sigma),\pi^{L}(\beta_{0},\boldsymbol{\beta}_{p},\sigma)\big)$ the
Bayes factor of $M_{p}$ against $M_{0}$ is
\begin{equation}
B_{p0}^{L}(\mathbf{y},\mathbf{X})=\frac{1}{2n}\int_{0}^{\infty}\!\frac{(1+g)^{(n-p-1)/2}}{(1+g\mathcal{B}_{p0})^{(n-1)/2}}\left(
1+g/n \right)^{\!-3/2}dg. \label{LBF}
\end{equation}

\noindent The integral in (\ref{LBF}) does not have a closed form expression.

\subsubsection{The prior $\pi ^{CG}(g \,|\, n)$}

A quite close prior to the preceding one was introduced by \cite{Cui2008} and its expression is
$$
\pi^{CG}(g \,|\, n) = (1+g)^{-2}.
$$
The Bayes factor for this prior turns out to be
\begin{equation}
B_{p0}^{CG}(\mathbf{y},\mathbf{X}) = \int_{0}^{\infty}
\frac{(1+g)^{(n-p-1)/2}}{(1+g\mathcal{B}_{p0})^{(n-1)/2}} \, (1+g)^{-2} dg. \label{CGBF}
\end{equation}
The integral in (\ref{CGBF}) does not have a close form expression.

\subsubsection{The prior $\pi^{B}(g \,|\, n,p)$} \smallskip

Later on \cite{Bayarri2012} also corrected the inconsistency of \smallskip the Bayes factor for
the prior $\pi(g) = \frac{1}{2}(1+g)^{-3/2}$ truncating this prior to the subset of the\smallskip
real line $\{g:g
> (1+n)/(1+p)-1\}$. The resulting prior for $g$ depends on $n$ and $p$ and is given by
$$
\pi^{B}(g \,|\, n,p) = \frac{1}{2} \, \left( \frac{1+n}{1+p} \right)^{1/2} (1+g)^{-3/2} \,
1_{\left( \frac{1+n}{1+p}-1,\infty \right)}(g).
$$
Thus, the prior for $(\beta_{0},\boldsymbol{\beta}_{p},\sigma)$ is
$$
\pi^{B}(\beta_{0},\boldsymbol{\beta}_{p},\sigma) = \dfrac{c}{2 \, \sigma} \,
\left( \dfrac{1+n}{1+p} \right)^{1/2} \int_{\frac{1+n}{1+p}-1}^{\infty} N(%
\boldsymbol{\beta}_{p} \mid \mathbf{0}_{p}, g\sigma^{2}(\mathbf{X}%
_{p}^{\prime} \mathbf{X}_{p})^{-1}) \, (1 + g)^{-3/2} dg ,
$$
and the Bayes factor of $M_{p}$ against $M_{0}$ for $\big(\pi^{N}(\beta_{0}, \sigma),
\pi^{B}(\beta_{0},\boldsymbol{\beta}_{p},\sigma) \big)$ is
\begin{equation}
B_{p0}^{B}(\mathbf{y},\mathbf{X}) = \frac{1}{2} \left( \frac{1 + n}{1 + p} \right)^{1/2}
\int_{\frac{1 + n}{1 + p} - 1}^{\infty} \frac{(1 + g)^{(n-p-1)/2}}{(1 +
g\mathcal{B}_{p0})^{(n-1)/2}} \left(1 + g \right)^{-3/2} dg. \label{BBF}
\end{equation}

\noindent The integral in (\ref{BBF}) does not have a closed form expression.

\subsubsection{The robust g-prior class}

We note that the priors $\pi^{L}(g)$, $\pi^{CG}(g \,|\, n)$ and $\pi^{B}(g \,|\, n,p)$ are
particular cases of the ``robust'' g-prior class proposed by \cite{Bayarri2012}. This class was
defined as
\begin{equation}
\pi^{R}(g \,|\ a,d,\rho) = a [\rho (d+n)]^{a} (g+d)^{-(a+1)} \mathbf{1}_{\{g > \rho (d+n) -
d\}}(g) \label{RBF}
\end{equation}
for $a,d>0$ and $\rho \geq d/(d+n)$.

\section{Asymptotic}\label{asymp}

Let us assume that $p=O(n^{b})$ for $0 \leq b \leq 1$. To facilitate the reading of this section
we bring here some auxiliary results on the asymptotic distribution of the sampling statistic
$\mathcal{B}_{p0}$ when $p$ grows with $n$. We also give a lower bound of the Bayes factor
$B_{p0}^{L}$ and an approximation to $B_{p0}^{B}$. These tools are given in Lemma \ref{AdistrBp0}
and \ref{LBnLarge}. The limit in probability of a random sequence when sampling from model $M$ is
denoted as $[P_{M}]$. \medskip

\begin{lemma} \label{AdistrBp0}
For $p=O(n^{b})$ with $0 \leq b \leq 1$ and the pseudo-distance between a sampling model in
$M_{p}$ and $M_{0}$
$$
\delta_{p0} = \frac{1}{2\sigma_{p}^{2}} \, \frac{\boldsymbol{\beta}_{p+1}^{\prime}
\mathbf{X}^{\prime} \left( \mathbf{H} - \frac{1}{n} \mathbf{1}_{n}^{\phantom{\prime}}
\mathbf{1}_{n}^{\prime} \right) \mathbf{X} \boldsymbol{\beta}_{p+1}}{n} ,
$$
we have that \medskip
\begin{enumerate}[i)]
\item for $0\leq b<1$,
$$
\lim_{n\rightarrow \infty }\mathcal{B}_{p0} = \left\{
\begin{array}{ll}
1, & [P_{M_{0}}], \\
(1 + \delta)^{-1}, & [P_{M_{p}}],
\end{array}
\right.
$$
where $\thinspace \delta = \lim_{n \rightarrow \infty} \delta _{p0} \geq 0$, \medskip

\item for $b=1$ and $\thinspace r = \lim_{n \rightarrow \infty} n/p > 1$,
$$
\lim_{n\rightarrow \infty }\mathcal{B}_{p0} = \left\{
\begin{array}{ll}
1-1/r, & [P_{M_{0}}], \\
(1-1/r)(1+\delta )^{-1}, & [P_{M_{p}}].
\end{array}
\right.
$$
\end{enumerate}
\end{lemma}

\begin{proof}
The proof is given in Lemma 1 in \cite{Moreno2010} and, hence, it is omitted.
\end{proof}

\subsection{The case $\mathbf{b<1}$}

When $n$ goes to infinity the Bayes factors $B_{p0}^{IP}$, $B_{p0}^{IPH}$, $B_{p0}^{ZS}$,
$B_{p0}^{FS}$ and $B_{p0}^{B}$ are all consistent. Further for $b=0$, that is for finite $p$, the
convergence rate under $M_{0}$ is $O(n^{-p/2})$ and exponential under $M_{p}$. However, for
$0<b<1$ the rate of convergence under $M_{0}$ is not $O(n^{-p/2})$ although it is exponential
under $M_{p}$. Some of the consistency proofs of the Bayes factors have been given by
\cite{Fernandez2001}, \cite{Liang2008}, \cite{Bayarri2012} and \cite{Moreno2010}. Thus we only
examine here the consistency of those Bayes factors for which either there is no consistency proof
or the existing ones are not correct. Hence, we pay attention to the Bayes factors $B_{p0}^{L}$,
$B_{p0}^{CG}$, $B_{p0}^{R}$, $B_{p0}^{IPH}$ and $B_{p0}^{B}$. \medskip

\begin{lemma} \label{LBnLarge}
For any $\mathcal{B}_{p0}$ and large $n$ we have that \smallskip

\begin{enumerate}[i)]
\item $B_{p0}^{L}(\mathcal{B}_{p0}) \geq \dfrac{1}{2} \, n^{-(p+3)/2} \,
\mathcal{B}_{p0}^{-\frac{n-1}{2}} \left( \dfrac{1 - \mathcal{B}_{p0}}{2 \mathcal{B}_{p0}}
\right)^{\negthinspace -\frac{p+1}{2}} \Gamma \left( \tfrac{p+1}{2} \right)$.
\medskip

\item $B_{p0}^{B}(\mathcal{B}_{p0}) \approx \dfrac{1}{2} \,
\dfrac{n^{-p/2}}{(p+1)^{1/2}} \, \mathcal{B}_{p0}^{-\frac{n-1}{2}} \left( \negthinspace
\dfrac{1-\mathcal{B}_{p0}}{2 \mathcal{B}_{p0}} \negthinspace \right)^{\negthinspace
-\frac{p+1}{2}} \, \gamma \negthinspace \left( \frac{p+1}{2},
\frac{1-\mathcal{B}_{p0}}{\mathcal{B}_{p0}} \frac{p+1}{2} \right)$, \smallskip

\noindent where $\gamma(a,b) = \displaystyle{\int_0^b x^{a-1} exp(-x) \, dx}$ is the lower
incomplete gamma function.
\end{enumerate}
\end{lemma}

\begin{proof}
See Appendix \ref{secA1}
\end{proof}

\vskip 10pt The next theorem shows that the Bayes factor $B_{p0}^{IPH}$ and $B_{p0}^{B}$ are
consistent, but $B_{p0}^{L}$ and the Bayes factors for a subclass of robust priors are
inconsistent for any $p$. \medskip

\begin{theorem} \label{asymptoticLRM0b0} $\ $ \medskip

\begin{enumerate}[i)]
\item The Bayes factor $B_{p0}^{L}(\mathcal{B}_{p0})$\ is inconsistent under the null model $M_{0}$ for any
$p$. \smallskip

\item The Bayes factor $B_{p0}^{R}(\mathcal{B}_{p0})$ for any robust g-prior in the class
$$
\mathcal{R} = \{ \pi^{R}(g \mid a,d,\rho) : a, d > 0, \rho = d/(d+n) \}
$$
is inconsistent under the null model $M_{0}$ for any $p$. \smallskip

\item The Bayes factors $B_{p0}^{IPH}(\mathcal{B}_{p0})$ and $B_{p0}^{B}(\mathcal{B}_{p0})$ are consistent.
\end{enumerate}

\end{theorem}

\begin{proof}
See Appendix \ref{secA2}
\end{proof}

\vskip 10pt
\begin{corollary}
The Bayes factor $B_{p0}^{CG}(\mathcal{B}_{p0})$ is inconsistent under the null model $M_{0}$ for
any $p$.
\end{corollary}

\begin{proof}
As $B_{p0}^{CG}(\mathcal{B}_{p0})$ is defined by the robust g-prior in the class $\mathcal{R}$ for
$a = 1$, $d = 1$ and $\rho = 1/(1+n)$, the assertion follows from ii) in Theorem 1.
\end{proof}

\subsection{The case $\mathbf{b=1}$}

For $b=1$ any of the Bayes factors are no longer consistent and every one is inconsistent when
sampling from a set of model $M_{p}$ that depends on the limit of $n/p$ and the limit of the
pseudo distance $\delta_{p0}$ between $M_0$ and $M_p$. The size of the inconsistency set varies
across the Bayes factors. These sets are given in the following theorem. \medskip

\begin{theorem} \label{asymptoticb1}
For \,$p=O(n)$, $\thinspace r = \lim_{n \rightarrow \infty} n/p > 1$ and $\thinspace \delta =
\lim_{n\rightarrow \infty} \delta _{p0} > 0$ we have: \medskip

\begin{enumerate}[i)]
\item The Bayes factor $B_{p0}^{IP}(\mathcal{B}_{p0})$ is consistent except
when sampling from $M_{p}$ and $(r,\delta)$ is in the set
$$
S^{IP} = \left\{ (r,\delta) : (1+r)^{r-1} \left( 1+ \frac{r-1}{1 + \delta} \right)^{-r} \leq 1
\right\}.
$$

\item The Bayes factor $B_{p0}^{IPH}(\mathcal{B}_{p0})$ is consistent except
when sampling from $M_{p}$ and $(r,\delta)$ is in the set
$$
S^{IPho} = \left\{ (r,\delta) : (1+2r)^{r-1} \left( 1+ \frac{2(r-1)}{1 + \delta} \right)^{-r} \leq
1 \right\}.
$$

\item The Bayes factor $B_{p0}^{ZS}(\mathcal{B}_{p0})$ is consistent except
when sampling from $M_{p}$ and $(r,\delta)$ is in the set
$$
S^{ZS} = \left\{ (r,\delta) : \frac{1}{r \, exp(1)} \left( \frac{1 + \delta}{1-1/r} \right)^{r-1}
\leq 1 \right\}.
$$

\item The Bayes factors $B_{p0}^{FS}(\mathcal{B}_{p0})$ is inconsistent when
sampling from any alternative model $M_{p}$. \smallskip

\item The Bayes factor $B_{p0}^{B}(\mathcal{B}_{p0})$ is consistent except
when sampling from $M_{p}$ and $(r,\delta)$ is in the set
$$
S^{B} = \left\{ (r,\delta) : \left( \dfrac{r}{r-1} \right)^{r-1} \frac{(1+\delta)^r}{1+\delta r}
\, exp(-1) \leq 1 \right\}.
$$
\end{enumerate}
\end{theorem}

\begin{proof}
See Appendix \ref{secA3}
\end{proof}

\section{Posterior model probability for moderate sample size}\label{pprob}

Given a Bayes factor $B_{p0}^{C}(\mathcal{B}_{p0})$ and the model prior $\pi (M_{0})=\pi
(M_{p})=0.5$, the posterior probability of model $M_{0}$ in the set $\{M_{0},M_{p}\}$ is given by
\[
P^{C}(M_{0}|\mathcal{B}_{p0}) = \left( 1 + B_{p0}^{C}(\mathcal{B}_{p0})\right) ^{-1}.
\]
Figure \ref{fig1} displays the posterior probability of $M_{0}$ as a function of
$\mathcal{B}_{p0}$ for those Bayes factors that are consistent for finite $p$, say $B_{p0}^{IP} $,
$B_{p0}^{IPH}$, $B_{p0}^{ZS}$, $B_{p0}^{FS}$ and $B_{p0}^{B}$, and some values of $n$ and $p$.
This Figure illustrates the fact that for small $p$ all posterior probability curves are very
similar, more so as $n$ increases. However, as $p$ grows the curves can be classified into two
different groups, one with the Bayes factors $\{B_{p0}^{ZS}, B_{p0}^{FS}\}$ and another with
$\{B_{p0}^{IP}, B_{p0}^{IPH}, B_{p0}^{B}\}$. The former group contains curves much less smoother
than the latter, that is much more in favor of model $M_{0}$ than the latter, more so as $p$
grows. Except for small $p$ and even for small $n$ the curves of the former group jump from almost
$0$ to almost $1$ in a very narrow interval. This is a behavior that one expects only for large
values of $n$.

\begin{figure}[H]
\centering
\begin{tabular}{lll}
\includegraphics[width=0.3\textwidth]{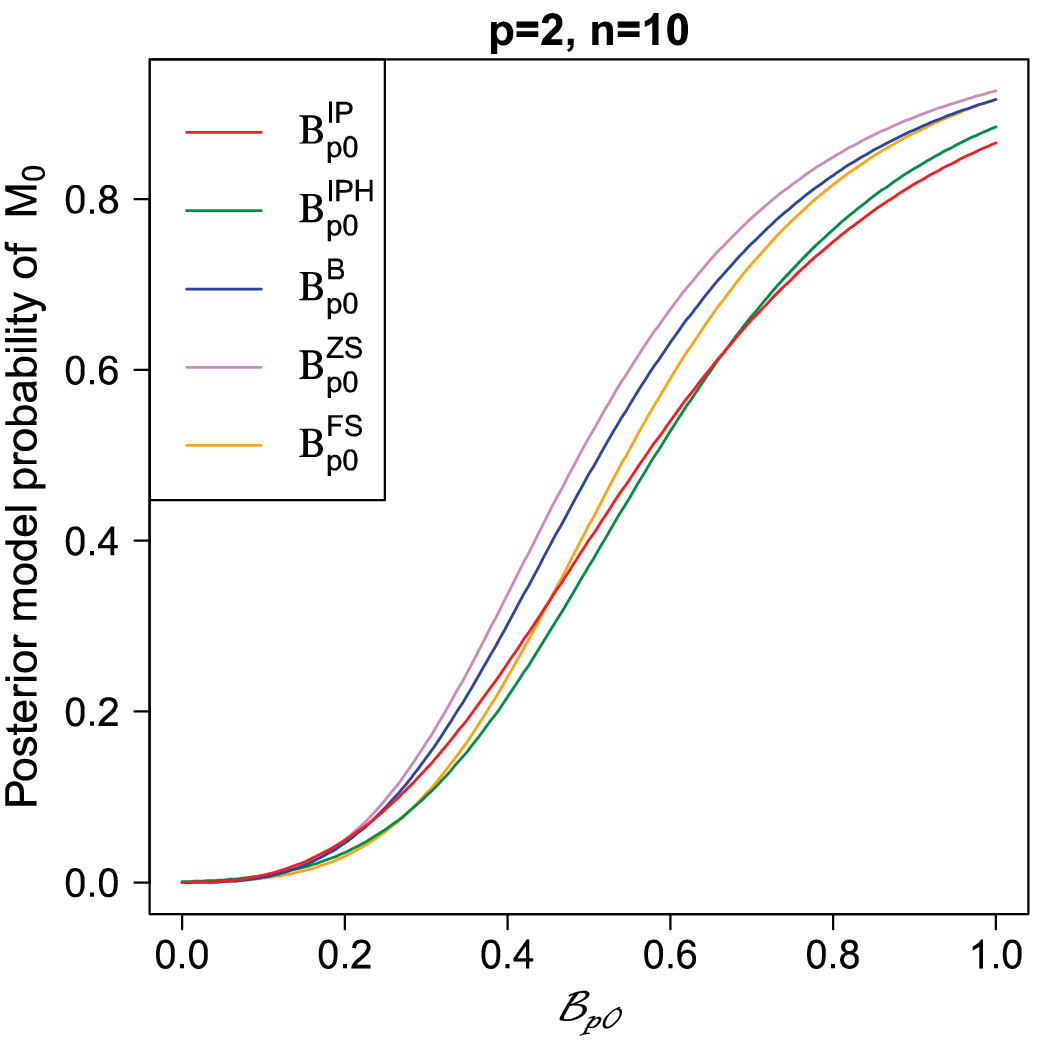} & \includegraphics[width=0.3\textwidth]{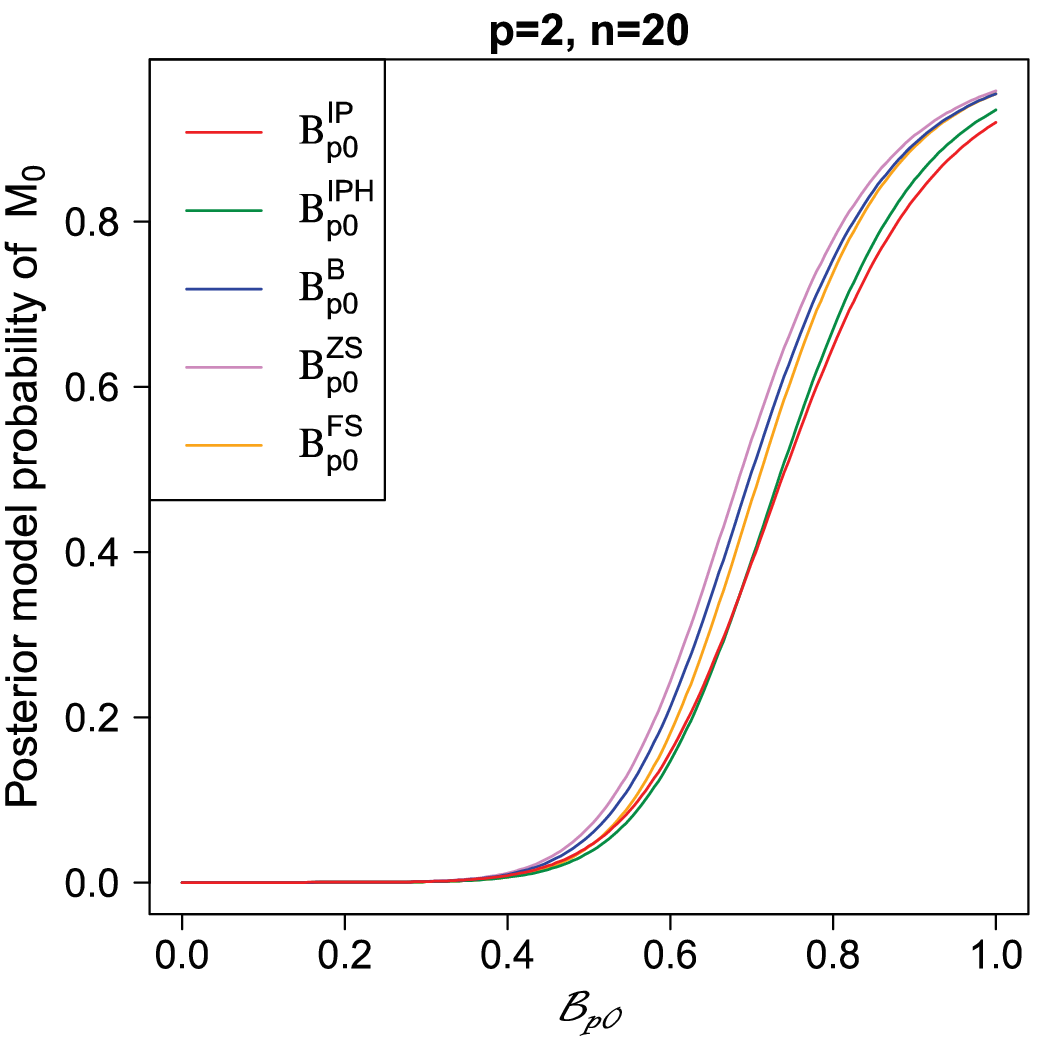} & \includegraphics[width=0.3\textwidth]{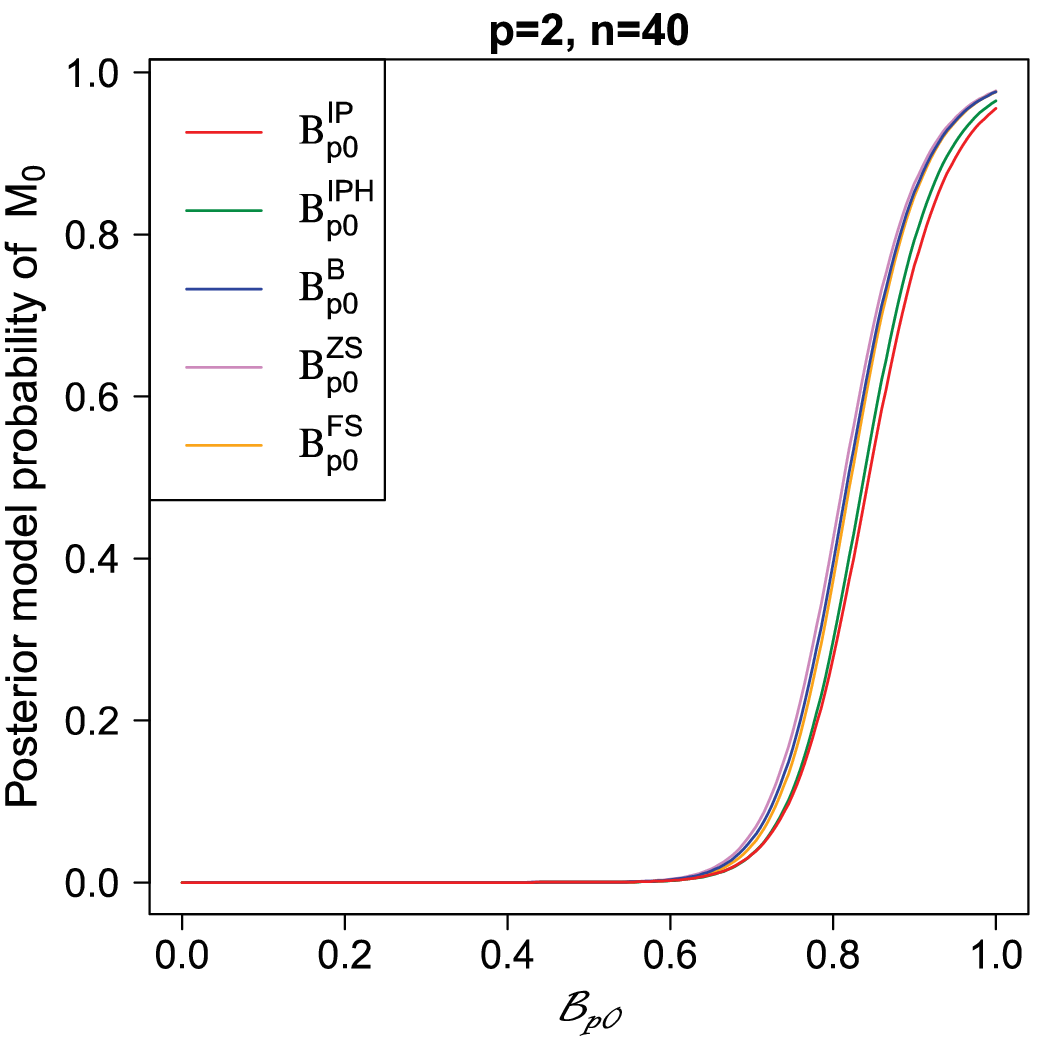} \\
\includegraphics[width=0.3\textwidth]{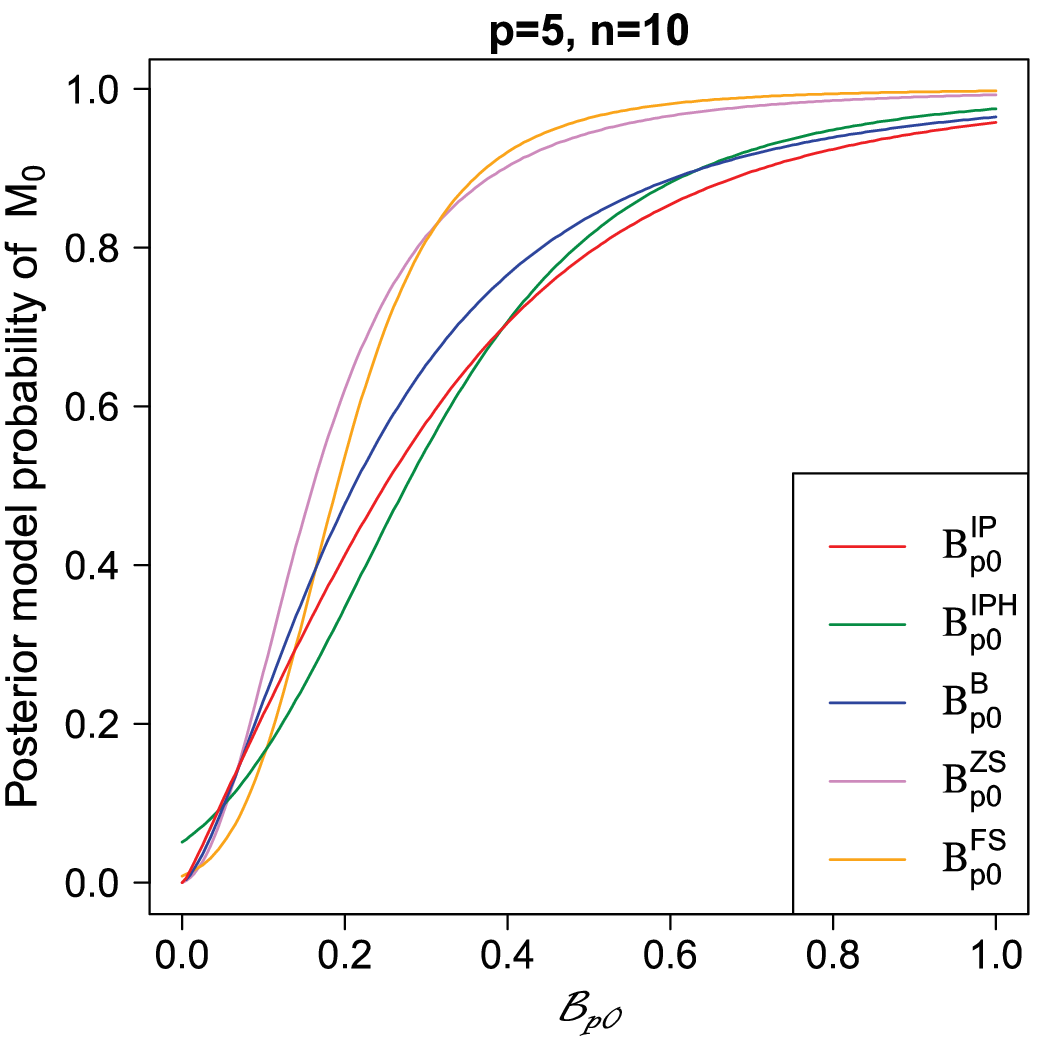} & \includegraphics[width=0.3\textwidth]{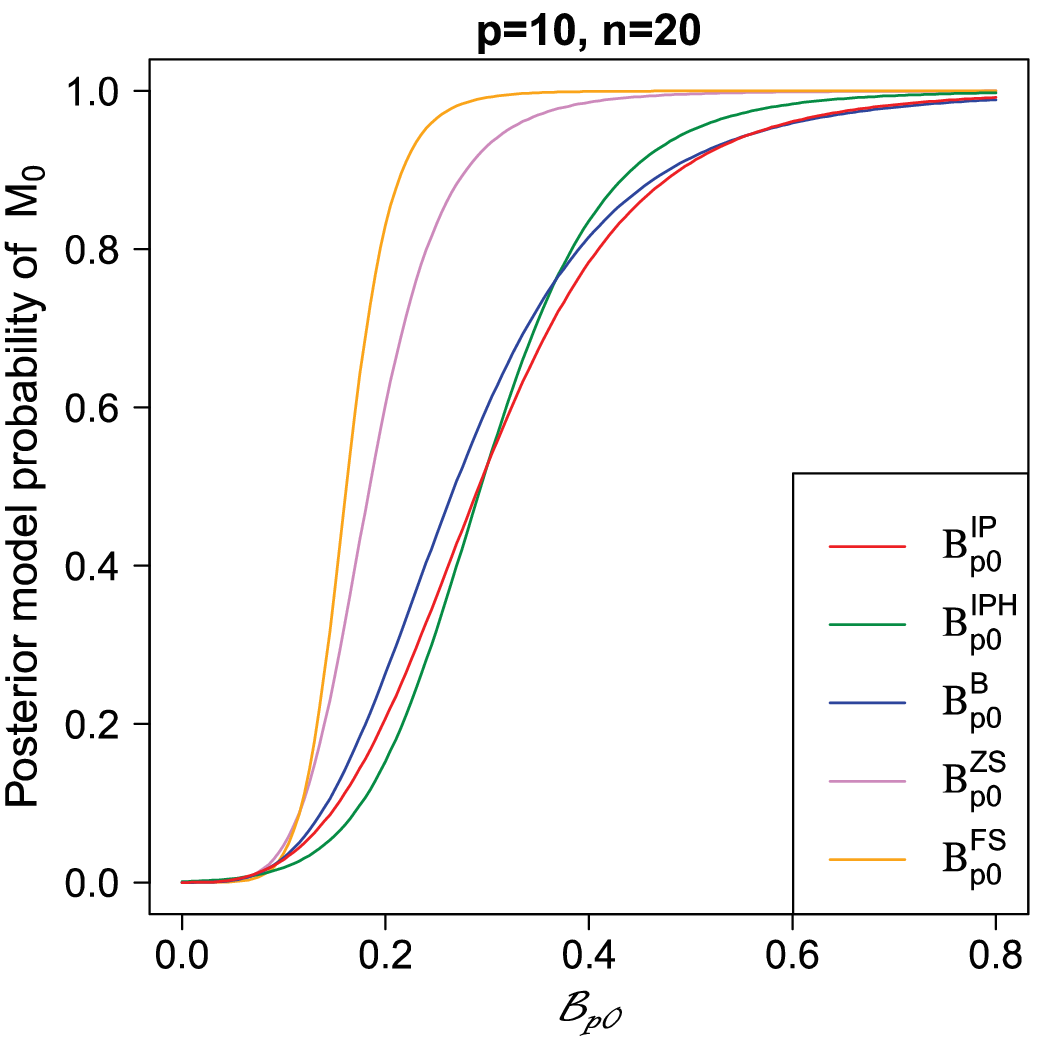} & \includegraphics[width=0.3\textwidth]{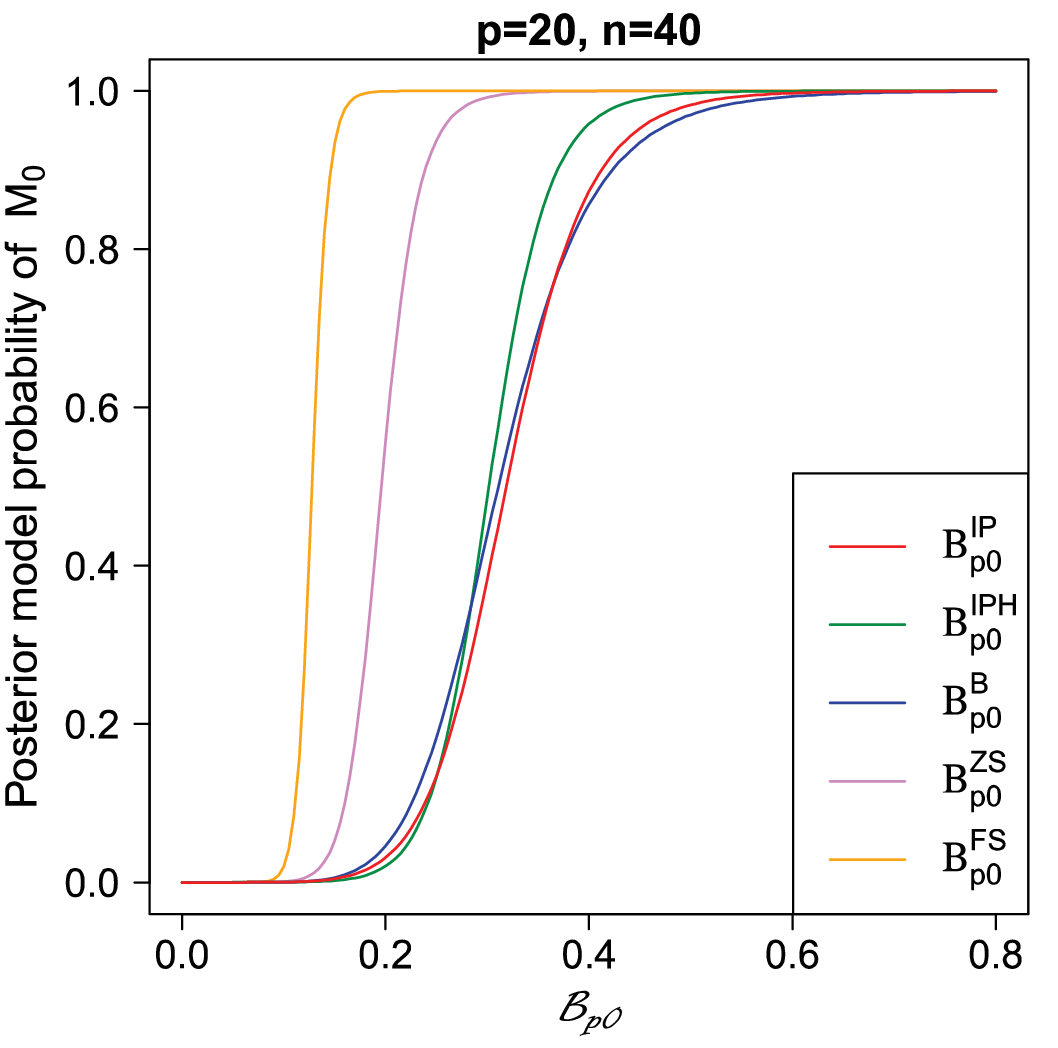} \\
\includegraphics[width=0.3\textwidth]{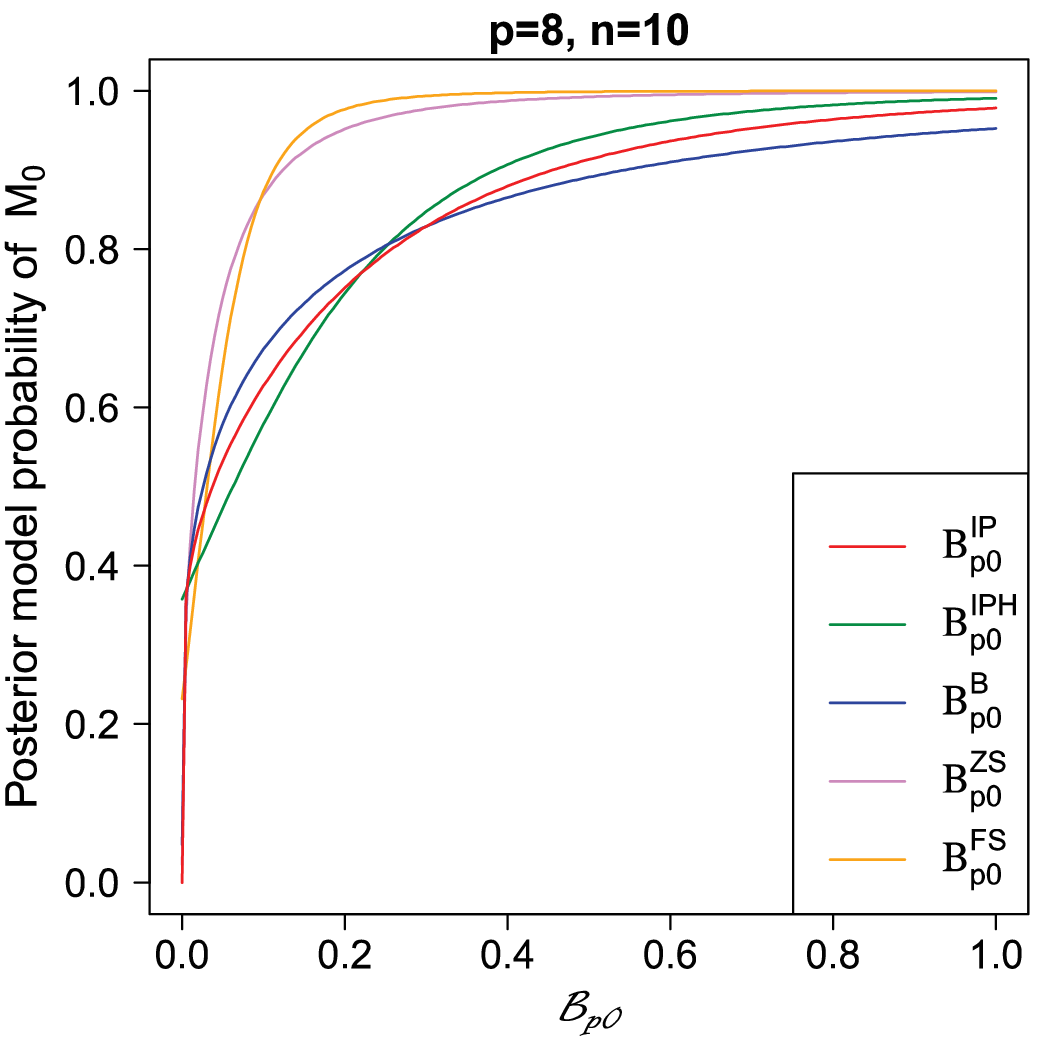} & \includegraphics[width=0.3\textwidth]{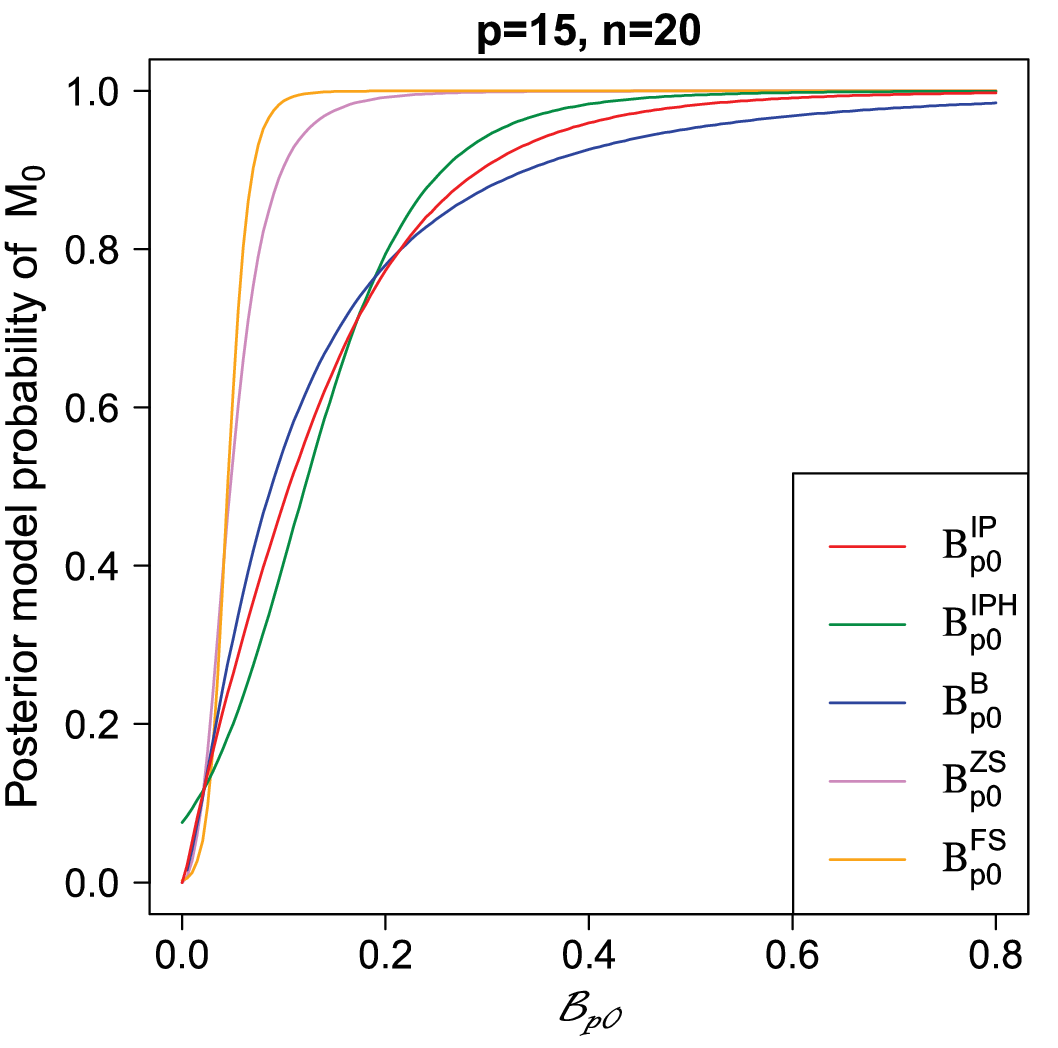} & \includegraphics[width=0.3\textwidth]{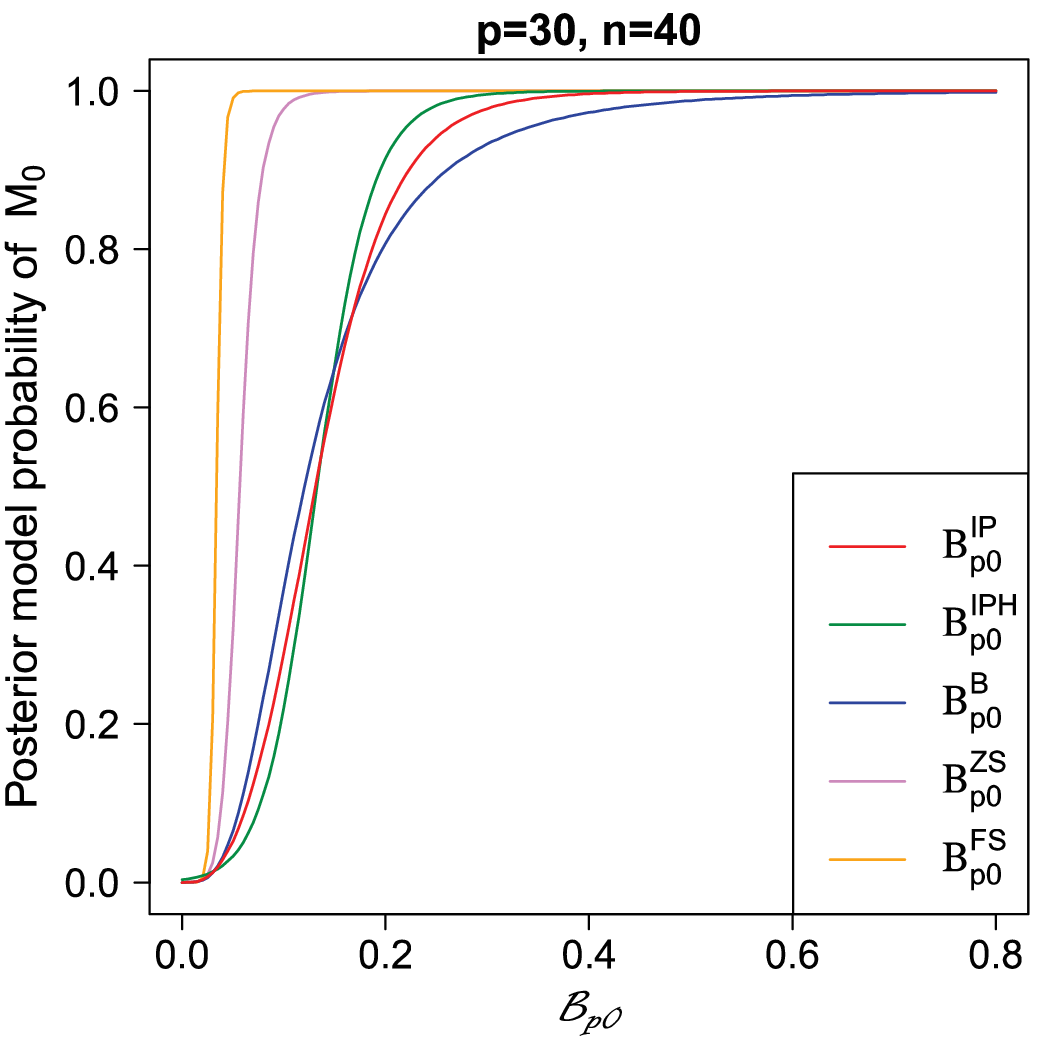} \\
\end{tabular}
\caption{Posterior probability of model $M_{0}$ in the set $\{M_{0},M_{p}\}$ as a function of
$\mathcal{B}_{p0}$.\label{fig1}}
\end{figure}

\section{Conclusion}\label{conc}

Asymptotically we have found the unexpected result that the Bayes factor
$B_{p0}^{L}(\mathcal{B}_{p0})$, $B_{p0}^{CG}(\mathcal{B}_{p0})$ and those for a wide subclass of
the robust prior class are inconsistent for any $p$. Further, for the robust priors in expression
(\ref{RBF}) the following conditions (i) $a$ and $\rho $ do not depend on $n$, (ii)
$\lim_{n\rightarrow \infty }b/n=c\geq 0$, (iii) $\lim_{n\rightarrow \infty }\rho (b+n)=\infty $,
and (iv) $n\geq p+1+2a$, were stated as sufficient conditions to ensure the consistency of the
Bayes factor \citep{Bayarri2012}. Unfortunately, this assertion is not true and the Bayes factor
$B_{p0}^{L}$, which is derived from the robust prior $\pi ^{R}(g|1/2,n,1/2)$, serves as a
counter-example.
\medskip

The others Bayes factors $B_{p0}^{IP}$, $B_{p0}^{IPH}$, $B_{p0}^{ZS}$, $B_{p0}^{FS}$ and
$B_{p0}^{B}$ are all consistent for $p=O(n^{b})$ and $0\leq b<1$. \medskip

For $b=1$, the latter Bayes factors are also consistent under the null. However, under the
alternative $B_{p0}^{FS}$ is inconsistent, so that it has the same asymptotic as the BIC
approximation, and the remaining Bayes factors are consistent except for specific sets of points
$(r,\delta)$. These inconsistency sets can be written in the explicit form
$$
S^{C}=\left\{ (r,\delta ):\delta <\delta ^{C}(r)\right\}
$$
for $C=IP, \ IPH, \ ZS$, where
$$
\delta ^{C}(r) = \left\{
\begin{array}{ll}
\dfrac{(r-1)}{(1+r)^{(r-1)/r} - 1} - 1, & \hbox{\ if $C = IP$,} \smallskip \\
\dfrac{2(r-1)}{(1+2r)^{(r-1)/r} - 1} - 1, & \hbox{\ if $C = IPH$,} \smallskip \\
(r \, exp(1))^{1/(r-1)} (1-1/r) - 1 , & \hbox{\ if $C = ZS$,} \\
\end{array}
\right.
$$
are decreasing convex functions of $r$ such that $\lim_{r \rightarrow \infty }\delta ^{C}(r)=0$,
so that as $r$ increases the inconsistency set of every Bayes factor tends to be the empty set.
\medskip

\begin{figure}[h]
\centering
\begin{tabular}{ll}
\includegraphics[width=0.4\textwidth]{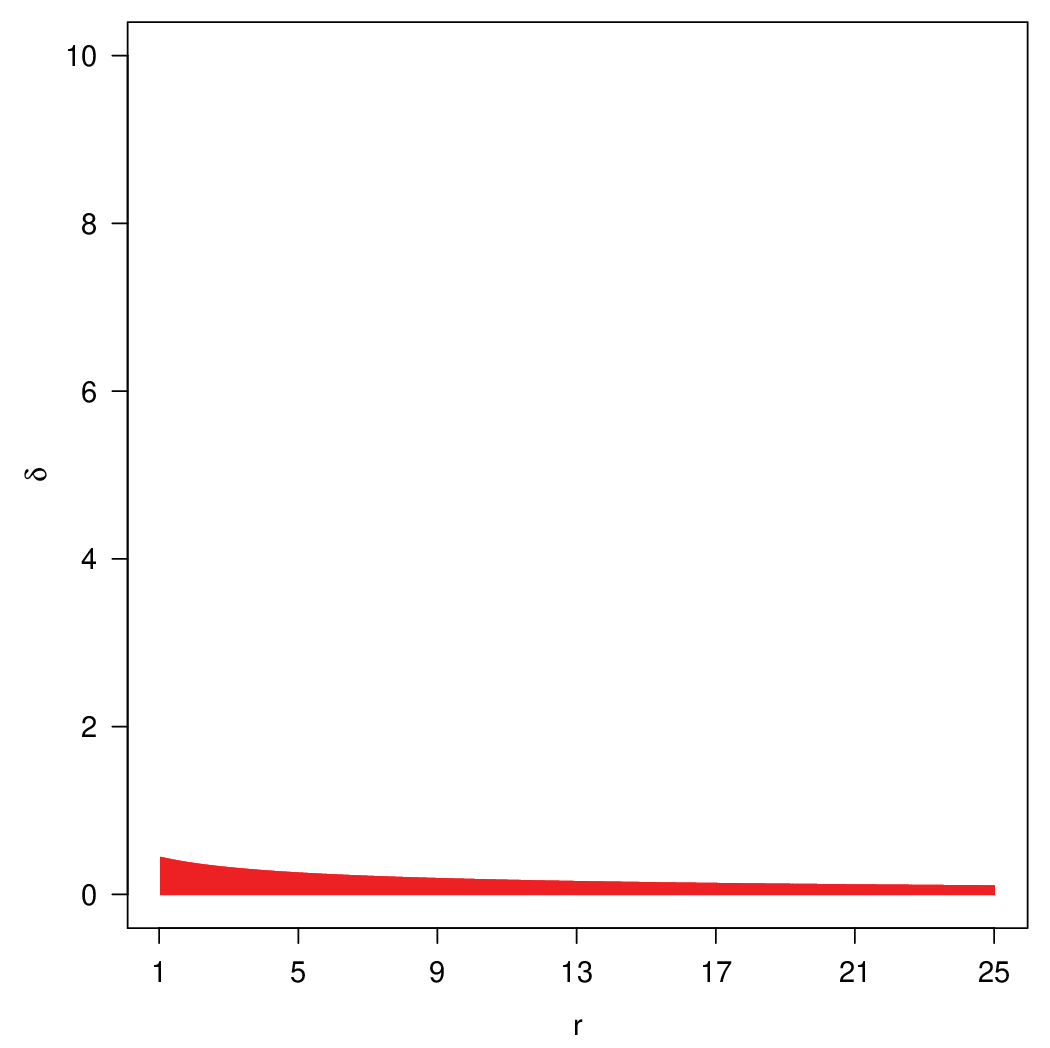} & \includegraphics[width=0.4\textwidth]{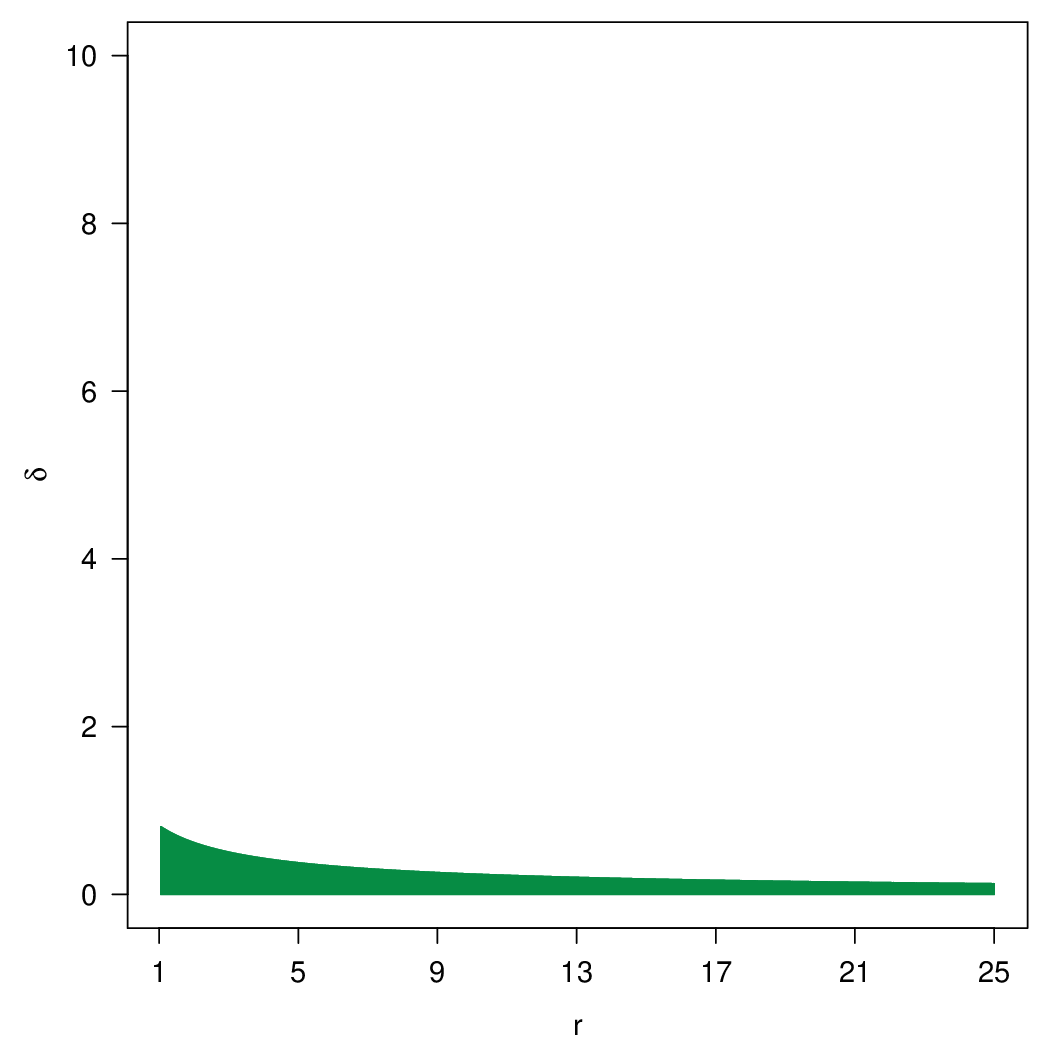} \\
\includegraphics[width=0.4\textwidth]{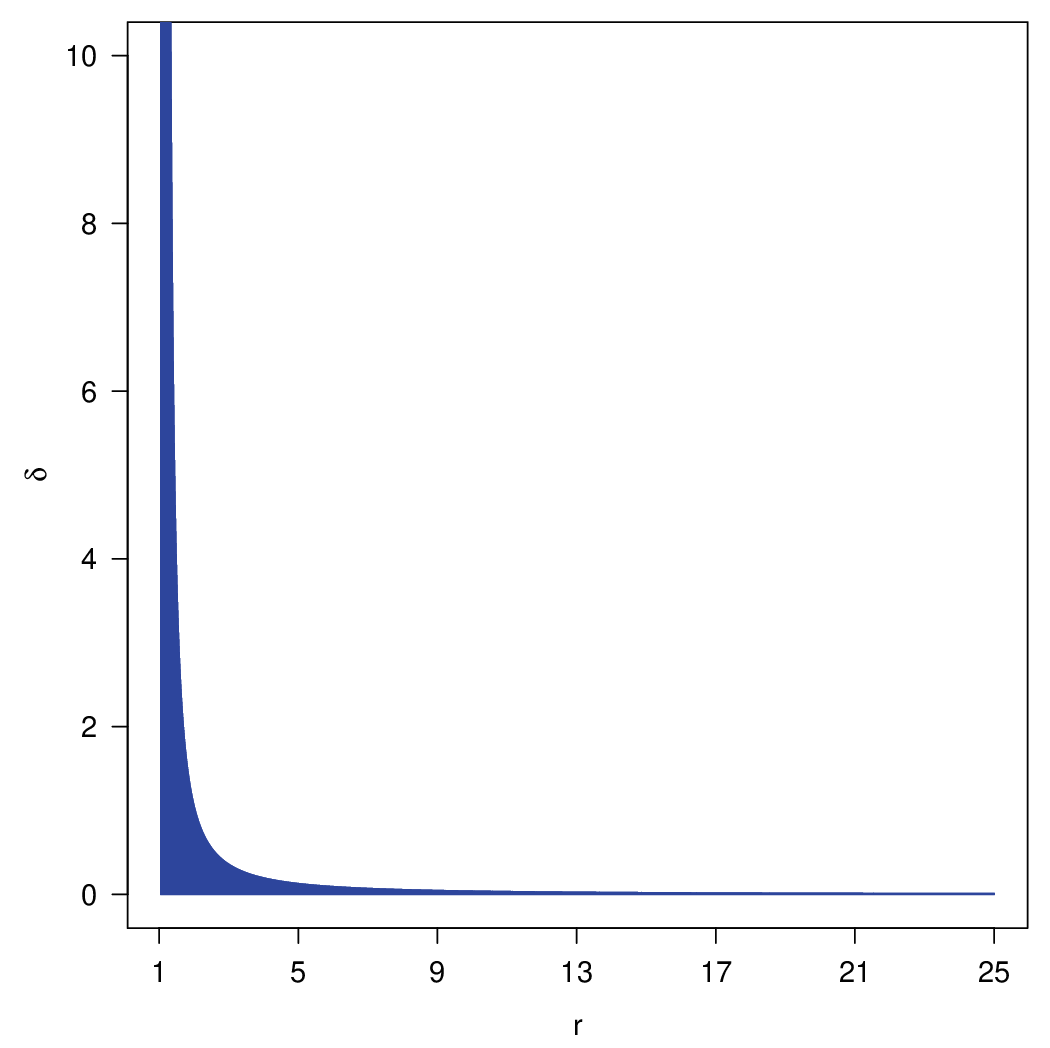} & \includegraphics[width=0.4\textwidth]{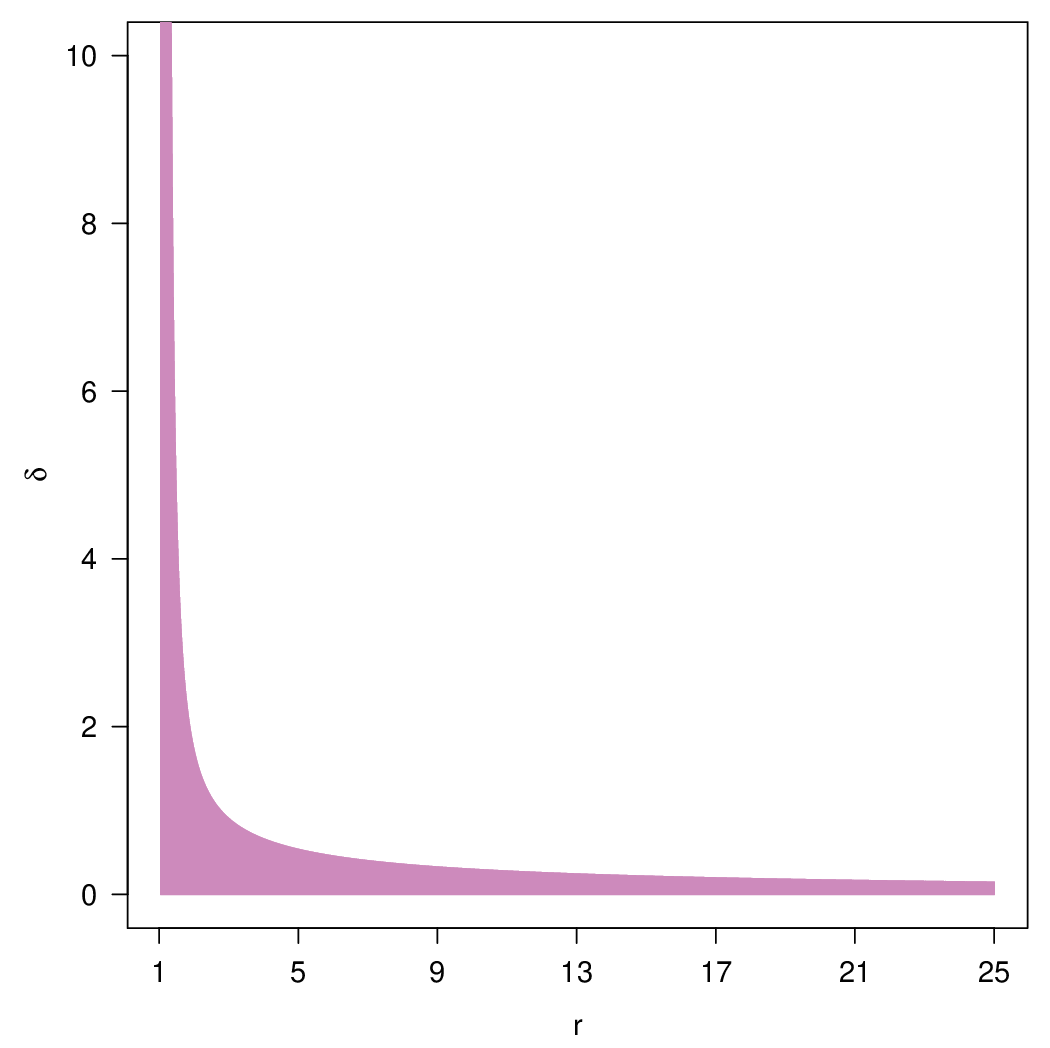} \\
\end{tabular}
\caption{Inconsistency set of $B_{p0}^{IP}$, top left panel, $B_{p0}^{IPH}$, top right panel,
$B_{p0}^{B}$, bottom left panel, and $B_{p0}^{ZS}$, bottom right panel. \label{fig2}}
\end{figure}

In Figure \ref{fig2} the inconsistency set of each Bayes factor is represented by a colored area.
Certainly the smaller the inconsistency set the preferred the Bayes factor as model selector. We
note that $S^{B} \subset S^{ZS}$ and hence $B_{p0}^{B} \succsim B_{p0}^{ZS}$. Also $S^{IP} \subset
S^{IPH} \subset S^{ZS}$ and hence $B_{p0}^{IP} \succsim B_{p0}^{IPH} \succsim B_{p0}^{ZS}$.
Moreover,
\[
\lim_{r\rightarrow 1}\delta ^{C}(r)=\left\{
\begin{array}{ll}
0.4427, & \hbox{\ if $C = IP$,} \\
0.8205, & \hbox{\ if $C = IPH$,} \\
\infty , & \hbox{\ if $C = ZS, \, B$,}
\end{array}
\right.
\]
which means that $S^{IP}$ and $S^{IPH}$ only includes alternative models close to the null
($\delta < 0.4427$ and $\delta < 0.8205$, respectively) while $S^{ZS}$ and $S^{B}$ include
alternative models that are spread out in the quadrant $r>1$, $\delta \geq 0$, and very far from
the null. Therefore, our overall conclusion is to recommend the Bayes factor $B_{p0}^{IP}$.

\begin{appendices}

\section{Proof of Lemma \ref{LBnLarge}} \label{secA1}

\begin{enumerate}[i)]
\item For $g \geq 0$ and $n \geq 1$ we have that $\left( 1+\frac{g}{n} \right)^{-3/2} \geq (1+g)^{-3/2}$ and then we can write
$$
\arraycolsep 2pt
\begin{array}{lll}
B_{p0}^{L}(\mathcal{B}_{p0}) & \geq & \displaystyle{\dfrac{1}{2n} \, \int_{0}^{\infty}
\frac{(1+g)^{(n-p-1)/2}}{(1+g\mathcal{B}_{p0})^{(n-1)/2}} \ (1+g)^{-3/2} dg} \\
\\ & = & \displaystyle{\frac{1}{2n} \, \mathcal{B}_{p0}^{-(n-1)/2} \ \int_{0}^{\infty} (1+g)^{-(p+3)/2} \left(
\negthinspace 1 + \frac{\mathcal{B}_{p0}-1}{1+g\mathcal{B}_{p0}} \right)^{\negthinspace (n-1)/2}
dg}.
\end{array}
$$
Changing the variable $g$ to $y=g/n$ we have
$$
\arraycolsep 2pt
\begin{array}{lll}
B_{p0}^{L}(\mathcal{B}_{p0}) & \geq & \displaystyle{\frac{1}{2} \, \mathcal{B}_{p0}^{-(n-1)/2}
\int_{0}^{\infty} (1+ny)^{-(p+3)/2} \left( 1 + \frac{\mathcal{B}_{p0}-1}{1 + ny \mathcal{B}_{p0}}
\right)^{\negthinspace (n-1)/2} dy} \\
\\ & = & \displaystyle{\frac{1}{2} \, n^{-(p+3)/2} \, \mathcal{B}_{p0}^{-(n-1)/2} \negthinspace
\int_{0}^{\infty} \negthinspace \left( \negthinspace y + \frac{1}{n} \right)^{\negthinspace
\negthinspace -(p+3)/2} \left( \negthinspace 1 + \frac{\mathcal{B}_{p0}-1}{1 + ny\mathcal{B}_{p0}}
\right)^{\negthinspace (n-1)/2} dy}.
\end{array}
$$
For large $n$ we have that the integrand in the last expression can be approximated by
$$
y^{-(p+3)/2} exp\left(-\frac{1 - \mathcal{B}_{p0}}{2 y \mathcal{B}_{p0}}\right)
$$
and hence
$$
B_{p0}^{L}(\mathcal{B}_{p0}) \geq \frac{1}{2} \, n^{-(p+3)/2} \, \mathcal{B}_{p0}^{-(n-1)/2}
\int_{0}^{\infty} y^{-(p+3)/2} exp \negthinspace \left( -\frac{1 - \mathcal{B}_{p0}}{2 y
\mathcal{B}_{p0}} \right) dy.
$$
Therefore, we have
$$
B_{p0}^{L}(\mathcal{B}_{p0}) \geq \frac{1}{2} \, n^{-(p+3)/2} \, \mathcal{B}_{p0}^{-(n-1)/2}
\left( \frac{1-\mathcal{B}_{p0}}{2 \mathcal{B}_{p0}} \right)^{\negthinspace -(p+1)/2} \Gamma
\negthinspace \left( \tfrac{p+1}{2} \right)
$$
and i) is proved. \medskip

\item The Bayes factor $B_{p0}^{B}(\mathcal{B}_{p0})$ can be written as
$$
B_{p0}^{B}(\mathcal{B}_{p0}) = \frac{1}{2} \left( \frac{1+n}{1+p} \right)^{\negthinspace 1/2}
\mathcal{B}_{p0}^{-(n-1)/2} \int_{\frac{1+n}{1+p} - 1}^{\infty} (1+g)^{-(p+3)/2} \left(
\negthinspace 1 + \frac{\mathcal{B}_{p0}-1}{1+g\mathcal{B}_{p0}} \right)^{\negthinspace (n-1)/2}
dg.
$$
Changing the variable $g$ to $y=g/n$ we have
$$
B_{p0}^{B}(\mathcal{B}_{p0}) = \frac{1}{2} \left( \frac{1+n}{1+p} \right)^{\negthinspace 1/2} n \
\mathcal{B}_{p0}^{-(n-1)/2} \negthinspace \int_{\frac{1-p/n}{1+p}}^{\infty} (1 + ny)^{-(p+3)/2}
\left( \negthinspace 1 + \frac{\mathcal{B}_{p0} - 1}{1 + n y \mathcal{B}_{p0}} \negthinspace
\right)^{\negthinspace (n-1)/2} \negthinspace dy.
$$
For large $n$ it follows
$$
\begin{array}{lll}
B_{p0}^{B}(\mathcal{B}_{p0}) & \approx & \dfrac{1}{2} \, \dfrac{n^{-p/2}}{(p+1)^{1/2}} \,
\mathcal{B}_{p0}^{-(n-1)/2} \displaystyle{\int_{\frac{1}{p+1}}^{\infty} \negthinspace y^{-(p+3)/2}
exp \negthinspace \left( - \frac{1 - \mathcal{B}_{p0}}{2 y \mathcal{B}_{p0}} \right) dy} \\ \\
& = & \dfrac{1}{2} \, \dfrac{n^{-p/2}}{(p+1)^{1/2}} \, \mathcal{B}_{p0}^{-(n-1)/2}
\displaystyle{\int_{0}^{p+1} z^{(p+3)/2 - 2} \ exp \left( -\frac{1 - \mathcal{B}_{p0}}{2
\mathcal{B}_{p0}} \thinspace z \right) dz}
\\ \\ & = & \dfrac{1}{2} \, \dfrac{n^{-p/2}}{(p+1)^{1/2}} \, \mathcal{B}_{p0}^{-(n-1)/2}
\left(\dfrac{1-\mathcal{B}_{p0}}{2 \mathcal{B}_{p0}} \right)^{-(p+1)/2} \gamma \negthinspace
\left( \frac{p+1}{2}, \frac{1-\mathcal{B}_{p0}}{\mathcal{B}_{p0}} \frac{p+1}{2} \right).
\end{array}
$$

\noindent This proves ii) and completes the proof of Lemma \ref{LBnLarge}.
\end{enumerate}

\section{Proof of Theorem \ref{asymptoticLRM0b0}} \label{secA2}

\begin{enumerate}[i)]
\item To prove the inconsistency of $B_{p0}^{L}(\mathcal{B}_{p0})$ under the null $M_{0}$ we use part i) in
Lemmas \ref{AdistrBp0} and \ref{LBnLarge} and we can write
$$
\begin{array}{lll}
\displaystyle{\lim_{n \rightarrow \infty} B_{p0}^{L}(\mathcal{B}_{p0})} & \geq &
\displaystyle{\lim_{n \rightarrow \infty} \ \frac{1}{2} \, n^{-(p+3)/2} \,
\mathcal{B}_{p0}^{-(n-1)/2} \left( \dfrac{1-\mathcal{B}_{p0}}{2 \mathcal{B}_{p0}} \right)^{
-(p+1)/2} \Gamma \negthinspace \left( (p+1)/2 \right)}
\\ \\ & = & 2^{(p-1)/2} \, \Gamma \negthinspace \left( (p+1)/2 \right) \, \displaystyle{\lim_{n \rightarrow
\infty} \ n^{-(p+3)/2} (1 - \mathcal{B}_{p0})^{-(p+1)/2}}
\\ \\ & = & \left\{ \begin{array}{ll}
1, & \hbox{ for $p=1$, $[P_{M_{0}}]$,} \smallskip \\
\infty, & \hbox{ for $p>1$, $[P_{M_{0}}]$,}
\end{array}
\right.
\end{array}
$$
where the last equality follows from the fact that $1 - \mathcal{B}_{p0}$ converges to zero in
probability $[P_{M_0}]$ at rate $O(n^{-2})$. This proves i). \medskip

\item For any prior in $\mathcal{R}$ the Bayes factor becomes
$$
\begin{array}{lll}
B_{p0}^{R}(\mathcal{B}_{p0}) & = & a d^{a} \displaystyle{\int_{0}^{\infty}
\frac{(1+g)^{(n-p-1)/2}}{(1+g\mathcal{B}_{p0})^{(n-1)/2}} \thinspace (g+d)^{-(a+1)} dg} \\ \\ & =
& a d^{a} \, \mathcal{B}_{p0}^{-(n-1)/2} \displaystyle{\int_{0}^{\infty} \left( 1 +
\frac{\mathcal{B}_{p0}-1}{1 + g\mathcal{B}_{p0}} \right)^{(n-1)/2} (1 + g)^{-p/2}} \,
(g+d)^{-(a+1)} \thinspace dg.
\end{array}
$$
Changing the variable $g$ to $y=g/n$ we have
$$
\begin{array}{lll}
B_{p0}^{R}(\mathcal{B}_{p0}) & = & a d^{a} \, n \, \mathcal{B}_{p0}^{-(n-1)/2} \smallskip \\ & &
\displaystyle{\int_{0}^{\infty} \negthinspace \left( \negthinspace 1 + \frac{\mathcal{B}_{p0}-1}{1
+ ny \mathcal{B}_{p0}} \negthinspace \right)^{\negthinspace (n-1)/2}} (1 + ny)^{-p/2} \, (ny +
d)^{-(a+1)} \thinspace dy \\ \\
& = & a d^{a} \, n^{-(p/2 + a)} \thinspace \mathcal{B}_{p0}^{-(n-1)/2} \smallskip \\ & &
\displaystyle{\int_{0}^{\infty} \negthinspace \left( \negthinspace 1 + \frac{\mathcal{B}_{p0}-1}{1
\negthinspace + \negthinspace ny \mathcal{B}_{p0}} \right)^{\negthinspace\negthinspace (n-1)/2}
\left(\frac{1}{n} + y \right)^{\negthinspace\negthinspace -\frac{p}{2}}} \displaystyle{ \left(y +
\frac{d}{n} \right)^{\negthinspace\negthinspace -(a+1)} dy}
\end{array}
$$
For $n$ big enough we have that the integrand in the last expression can be approximated by
$$
y^{\raisebox{3pt}{$\scriptstyle -\big(\frac{p}{2}+a+1\big)$}} exp\left(-\frac{1 -
\mathcal{B}_{p0}}{2 y \mathcal{B}_{p0}}\right)
$$
and hence
$$
\begin{array}{lll}
B_{p0}^{R}(\mathcal{B}_{p0}) & \approx & a d^{a} \, n^{-(p/2 + a)} \, \mathcal{B}_{p0}^{-(n-1)/2}
\displaystyle{\int_{0}^{\infty} y^{-(p/2 + a + 1)} exp\left(-\frac{1 - \mathcal{B}_{p0}}{2 y
\mathcal{B}_{p0}}\right) dy} \\ \\ & = & a d^{a} \, n^{-(p/2 + a)} \, \mathcal{B}_{p0}^{-(n-1)/2}
\left( \dfrac{1-\mathcal{B}_{p0}}{2 \mathcal{B}_{p0}} \right)^{\negthinspace -(p/2 + a)} \Gamma
\negthinspace \left( p/2 + a \right).
\end{array}
$$
When sampling from the null $M_0$ the limit in probability of $B_{p0}^{R}(\mathcal{B}_{p0})$ turns
out to be
$$
\lim_{n \rightarrow \infty} B_{p0}^{R}(\mathcal{B}_{p0}) = \lim_{n \rightarrow \infty} \big( n
(1-\mathcal{B}_{p0}) \big)^{-(p/2+a)} = \infty, \, [P_{M_0}], \smallskip
$$
where the last equality follows from i) in Lemma \ref{AdistrBp0} and the fact that $1 -
\mathcal{B}_{p0}$ converges to zero in probability $[P_{M_0}]$ at rate $O(n^{-2})$. This proves
ii).
\medskip

\item To prove the consistency of $B_{p0}^{IPH}(\mathcal{B}_{p0})$ we rewrite the Bayes factor as
$$
B_{p0}^{IPH}(\mathcal{B}_{p0}) = \left( \negthinspace 1 + \frac{2n}{p+1}
\right)^{\negthinspace\negthinspace -p/2} \, \mathcal{B}_{p0}^{-(n-1)/2} \left( 1 - \frac{1 -
\mathcal{B}_{p0}}{1 + \frac{2n}{p+1} \, \mathcal{B}_{p0}} \right)^{\negthinspace (n-1)/2}.
$$
For large $n$ we have that
$$
B_{p0}^{IPH}(\mathcal{B}_{p0}) \approx \left( \negthinspace 1 + \frac{2n}{p+1}
\right)^{\negthinspace\negthinspace -p/2} \, \mathcal{B}_{p0}^{-(n-1)/2} \, exp \left( -\frac{1 -
\mathcal{B}_{p0}}{\mathcal{B}_{p0}} \, \frac{p+1}{4} \right). \smallskip
$$
When sampling from the $M_0$ we have from i) in Lemma \ref{AdistrBp0} that
$$
\lim_{n \rightarrow \infty} B_{p0}^{IPH}(\mathcal{B}_{p0}) = \lim_{n \rightarrow \infty} \left( 1
+ 2 n^{1-b} \right)^{\negthinspace - n^b/2} \, exp\left( - n^{b-2}/4 \right) = 0, \, [P_{M_0}].
$$
This proves the consistency under $M_{0}$. \medskip

\noindent When sampling from the alternative $M_p$ it follows from i) in Lemma \ref{AdistrBp0}
that
$$
\lim_{n \rightarrow \infty} B_{p0}^{IPH}(\mathcal{B}_{p0}) = \lim_{n \rightarrow \infty} \left( 1
+ 2 n^{1-b} \right)^{-n^b/2} (1 + \delta)^{(n-1)/2} \, exp\left(\negthinspace -\delta n^{b}/4
\right) = \infty, \, [P_{M_p}].
$$
This completes the proof of the consistency of $B_{p0}^{IPH}$. \bigskip

\noindent To prove the consistency of $B_{p0}^{B}$ when sampling from the $M_0$ we note that
$$
\lim_{n \rightarrow \infty} \frac{1-\mathcal{B}_{p0}}{\mathcal{B}_{p0}} \ \frac{p+1}{2} = 0, \,
[P_{M_0}],
$$
and hence we have for large $n$ that
$$
\gamma \negthinspace \left( \tfrac{p+1}{2}, \tfrac{1-\mathcal{B}_{p0}}{\mathcal{B}_{p0}} \,
\tfrac{p+1}{2} \right) \approx \frac{2}{p+1} \left( \frac{1-\mathcal{B}_{p0}}{\mathcal{B}_{p0}} \
\frac{p+1}{2} \right)^{(p+1)/2}.
$$
Therefore, from ii) in Lemma \ref{LBnLarge} it follows that
$$
\lim_{n \rightarrow \infty} B_{p0}^{B}(\mathcal{B}_{p0}) = \lim_{n \rightarrow \infty}
(p+1)^{p/2-1} \, n^{-p/2} = 0, \ [P_{M_0}],
$$
so that $B_{p0}^{B}(\mathcal{B}_{p0})$ is consistent under de null model $M_0$. \medskip

\noindent When sampling from the $M_p$, we distinguish the case $b=0$ and $0<b<1$. \smallskip

\noindent For $b=0$, we have
$$
\lim_{n \rightarrow \infty} \frac{1-\mathcal{B}_{p0}}{\mathcal{B}_{p0}} \ \frac{p+1}{2} =
\frac{\delta (p+1)}{2}, \, [P_{M_p}].
$$
Then, for large $n$,
$$
\gamma \negthinspace \left( \tfrac{p+1}{2}, \tfrac{1-\mathcal{B}_{p0}}{\mathcal{B}_{p0}} \,
\tfrac{p+1}{2} \right) \approx \gamma \negthinspace \left( \tfrac{p+1}{2}, \tfrac{\delta (p+1)}{2}
\right), \, [P_{M_p}],
$$
and from ii) in Lemma \ref{LBnLarge} we have that
$$
B_{p0}^{B}(\mathcal{B}_{p0}) \approx A(p,\delta) \, n^{-p/2} \, (1 + \delta)^{(n-1)/2}, \,
[P_{M_p}],
$$
where $A(p,\delta) = 2^{(p-1)/2} \, (p+1)^{-1/2} \, \delta^{-(p+1)/2} \, \gamma \negthinspace
\left( \tfrac{p+1}{2}, \tfrac{\delta (p+1)}{2} \right)$. \medskip

\noindent Therefore,
$$
\lim_{n \rightarrow \infty} \negthinspace B_{p0}^{B}(\mathcal{B}_{p0}) = A(p,\delta) \lim_{n
\rightarrow \infty} \negthinspace n^{-p/2} (1 + \delta)^{(n-1)/2} = \infty, \, [P_{M_p}].
$$
This proves the consistency of $B_{p0}^{B}$ for $b=0$. \medskip

\noindent For $0<b<1$ we have that
$$
\lim_{n \rightarrow \infty} \frac{1-\mathcal{B}_{p0}}{\mathcal{B}_{p0}} \ \frac{p+1}{2} = \infty,
\ [P_{M_p}],
$$
and then for large $n$
$$
\gamma \negthinspace \left(\tfrac{p+1}{2}, \tfrac{1-\mathcal{B}_{p0}}{\mathcal{B}_{p0}}
\tfrac{p+1}{2} \right) \approx \Gamma \negthinspace \left(\tfrac{p+1}{2} \right), \ [P_{M_p}].
$$
Therefore,
$$
\lim_{n \rightarrow \infty} B_{p0}^{B}(\mathcal{B}_{p0}) = \lim_{n \rightarrow \infty} A \, (1 +
\delta)^{(n-1)/2} = \infty, \ [P_{M_p}],
$$
where $A = 2^{(p-1)/2} \, (\delta n)^{-\frac{p}{2}} \ [\delta (p+1)]^{-1/2} \, \Gamma
\negthinspace \left(\tfrac{p+1}{2} \right)$. This completes the proof of the theorem.
\end{enumerate}

\section{Proof of Theorem \ref{asymptoticb1}} \label{secA3}

\begin{enumerate}[i)]
\item The proof of i) has been given in \cite{Moreno2010} and hence it is omitted. \medskip

\item When sampling from model $M_{0}$ we have
$$
\lim_{p\rightarrow \infty }B_{p0}^{IPH}(\mathcal{B}_{p0})=\lim_{p \rightarrow \infty } \left(
\frac{(1+2r)^{r-1}}{(-1+2r)^{r}}\right) ^{p/2}, \text{ }[P_{M_{0}}].
$$
The function of $r$ inside the bracket is smaller than $1$ for $r>1$ and this prove the
consistency under $M_{0}$. \medskip

\noindent When sample from $M_{p}$ and we have that
$$
\lim_{p\rightarrow \infty }B_{p0}^{IPH}(\mathcal{B}_{p0})=\lim_{p \rightarrow \infty} \left(
\frac{(1+2r)^{r-1}}{\left( 1+\frac{2(r-1)}{1+\delta }\right) ^{r}}\right) ^{p/2},[P_{M_{p}}].
$$
Thus, for any $(r,\delta )$ such that
$$
(1+2r)^{r-1}\left( 1+\frac{2(r-1)}{1+\delta }\right) ^{-r}\leq 1
$$
the Bayes factor is inconsistent under $M_{p}$. This proves the assertion. \medskip

\item Using the approximation of $B_{p0}^{ZS}(\mathcal{B}_{p0})$ for large $n$ in Lemma 2 in \cite{Moreno2015},
$$
B_{p0}^{ZS}(\mathcal{B}_{p0}) \approx \left( \frac{n \, exp(1)}{p+1} \right)^{-p/2} \left( \frac{1
- p/n}{1 + \delta_{p0}} \right)^{-(n-p-2)/2}, \smallskip
$$
when sampling from the null $M_{0}$ we have that $\displaystyle{\lim_{n\rightarrow \infty}}
\delta_{p0}=0$ and thus
$$
\lim_{n\rightarrow \infty} B_{p0}^{ZS}(\mathcal{B}_{p0}) = \lim_{p \rightarrow \infty} \ (r \,
exp(1))^{-p/2} \left( 1-\frac{1}{r} \right)^{-(p(r-1))/2} = 0, \ [P_{M_{0}}].
$$
This proves the consistency of $B_{p0}^{ZS}$ under $M_{0}$. \medskip

\noindent When sampling from the alternative $M_{p}$ we have that
$$
\begin{array}{lll}
\displaystyle{\lim_{n \rightarrow \infty} B_{p0}^{ZS}(\mathcal{B}_{p0})} & = &
\displaystyle{\lim_{p \rightarrow \infty} \ (r \, exp(1))^{-p/2} \left( \frac{1 - 1/r}{1 + \delta}
\right)^{-(p(r-1))/2}} \medskip \\ & = & \displaystyle{\lim_{p \rightarrow \infty} \left[
\frac{1}{r \, exp(1)} \left( \frac{1 + \delta}{1 - 1/r} \right)^{r-1} \right]^{p/2}}, \
[P_{M_{p}}].
\end{array}
$$
Therefore $B_{p0}^{ZS}(\mathcal{B}_{p0})$ is inconsistent if and only if
$$
\frac{1}{r \, exp(1)} \left( \frac{1 + \delta}{1-1/r} \right)^{r-1} \leq 1.
$$
This proves iii). \smallskip

\item For large $n$ we have that
$$
B_{p0}^{FS}(\mathcal{B}_{p0}) \approx n^{-p/2} \ \mathcal{B}_{p0}^{-(n-1)/2}
$$
When sampling from the null $M_0$ it follows that
$$
\lim_{n \rightarrow \infty} B_{p0}^{FS}(\mathcal{B}_{p0}) = \lim_{p \rightarrow \infty}
(pr)^{-p/2} (1 - 1/r)^{-pr/2} \lim_{p \rightarrow \infty} \left[ r \left( \negthinspace 1 -
\frac{1}{r} \right)^{\negthinspace r} p \right]^{-p/2} = 0, \ [P_{M_{0}}],
$$
and this proves the consistency under the null. \medskip

\noindent Under the alternative $M_p$ we have that
$$
\begin{array}{lll}
\displaystyle{\lim_{n \rightarrow \infty} B_{p0}^{FS}(\mathcal{B}_{p0})} & = &
\displaystyle{\lim_{p \rightarrow \infty} (pr)^{-p/2} \left[(1 - 1/r) (1 + \delta)^{-1}
\right]^{-pr/2}} \smallskip \\ & = & \displaystyle{\lim_{p \rightarrow \infty} \left[ r \left( 1 -
1/r \right)^{r} \negthinspace (1 + \delta)^{-r} p \right]^{-p/2} = 0}, \ [P_{M_{p}}],
\end{array}
$$
and hence $B_{p0}^{FS}(\mathcal{B}_{p0})$ is inconsistent for any $(r,\delta)$. This proves iv).
\medskip

\item When sampling from the null we have that ${\displaystyle \lim_{n \rightarrow \infty}}
\mathcal{B}_{p0} = 1-\dfrac{1}{r}$, $[P_{M_{0}}]$, and hence ${\displaystyle \lim_{n \rightarrow
\infty}} \dfrac{(1-\mathcal{B}_{p0})p}{2 \mathcal{B}_{p0}} = \infty$, $[P_{M_{0}}]$. Then, for
large $n$,
$$
\gamma \negthinspace \left(\tfrac{p+1}{2}, \tfrac{(1-\mathcal{B}_{p0})}{\mathcal{B}_{p0}}
\tfrac{p+1}{2} \right) \approx \Gamma \negthinspace \left(\tfrac{p+1}{2} \right), \ [P_{M_{0}}],
$$
and from part ii) in Lemma \ref{LBnLarge}
$$
\begin{array}{lll}
B_{p0}^{B}(\mathcal{B}_{p0}) & \approx & \dfrac{1}{2} \dfrac{n^{-p/2}}{p^{1/2}} \,
\mathcal{B}_{p0}^{-(n-1)/2} \left( \dfrac{1-\mathcal{B}_{p0}}{2 \mathcal{B}_{p0}}
\right)^{\negthinspace -(p+1)/2} \Gamma \negthinspace \left(\frac{p+1}{2} \right) \\ \\  & \approx
& \dfrac{1}{2} \dfrac{(p r)^{-p/2}}{p^{1/2}} \, \left( \dfrac{r-1}{r} \right)^{\negthinspace
-(pr-1)/2} \left( \dfrac{1}{2 (r-1)} \right)^{\negthinspace -(p+1)/2} \Gamma \negthinspace
\left(\frac{p+1}{2} \right) \\ \\ & = & A(r) \, 2^{(p-1)/2} \, p^{-(p+1)/2} \left( \dfrac{r}{r-1}
\right)^{\negthinspace (p(r-1))/2} \, \Gamma \negthinspace \left(\tfrac{p+1}{2} \right), \
[P_{M_{0}}],
\end{array}
$$
where $A(r) = r^{-1/2} (r-1) > 0$. \medskip

\noindent Further, from of the Stirling's approximation of the Gamma function, it follows that
$$
\begin{array}{lll}
B_{p0}^{B}(\mathcal{B}_{p0}) & \approx & A(r) \, \pi^{1/2} \, 2^{(p-1)/2} \, p^{-(p+1)/2} \left(
\dfrac{r}{r-1} \right)^{\negthinspace (p(r-1))/2} 2^{-(p-1)/2} \, p^{p/2} \, exp(-p/2) \\ \\ & = &
A(r) \, \pi^{1/2} \, p^{-1/2} \left( \dfrac{r}{r - 1} \right)^{\negthinspace (p(r-1))/2}
\negthinspace exp(-p/2), \ [P_{M_{0}}].
\end{array}
$$
Therefore,
$$
\lim_{n \rightarrow \infty} B_{p0}^{B}(\mathcal{B}_{p0}) = \lim_{p \rightarrow \infty} \left(
\negthinspace \left( \dfrac{r}{r - 1} \right)^{\negthinspace r - 1} exp(-1) \negthinspace
\right)^{\negthinspace p/2} \negthinspace = 0, \ [P_{M_{0}}],
$$
where the last equality follows from \vspace{1pt} $\big( \frac{r}{r-1} \big)^{r-1} exp(-1) < 1$.
This proves the consistency of $B_{p0}^{B}$ under the null model. \medskip

\noindent When sampling from the alternative we have that ${\displaystyle \lim_{n \rightarrow
\infty}} \mathcal{B}_{p0} = \dfrac{r-1}{r(1+\delta)}$, $[P_{M_{p}}]$, and hence ${\displaystyle
\lim_{n \rightarrow \infty}} \dfrac{(1-\mathcal{B}_{p0})p}{2 \mathcal{B}_{p0}} = \infty$,
$[P_{M_{p}}]$. Then, for large $n$,
$$
\gamma \negthinspace \left(\tfrac{p+1}{2}, \tfrac{(1-\mathcal{B}_{p0})p}{2 \mathcal{B}_{p0}}
\right) \approx \Gamma \negthinspace \left(\tfrac{p+1}{2} \right), \ [P_{M_{p}}],
$$
and from part ii) in Lemma \ref{LBnLarge}
$$
\begin{array}{lll}
B_{p0}^{B}(\mathcal{B}_{p0}) & \approx & \dfrac{1}{2} \dfrac{n^{-p/2}}{p^{1/2}} \,
\mathcal{B}_{p0}^{-(n-1)/2} \left( \dfrac{1-\mathcal{B}_{p0}}{2 \mathcal{B}_{p0}}
\right)^{\negthinspace -(p+1)/2} \Gamma \negthinspace \left(\frac{p+1}{2} \right) \\ \\  & \approx
& \dfrac{1}{2} \dfrac{(p r)^{-p/2}}{p^{1/2}} \, \left( \dfrac{r-1}{r(1+\delta)}
\right)^{\negthinspace -(pr-1)/2} \negthinspace \left( \negthinspace \dfrac{1+\delta r}{2 (r-1)}
\negthinspace \right)^{\negthinspace -(p+1)/2} \Gamma \negthinspace \left(\frac{p+1}{2} \right) \\
\\ & = & A(r,\delta) \, 2^{(p-1)/2} \, p^{-(p+1)/2} \left( \dfrac{r}{r - 1} \right)^{\negthinspace
(p(r-1))/2} \negthinspace \left( \dfrac{(1 + \delta)^{r}}{1 + \delta r} \right)^{\negthinspace
p/2} \Gamma \negthinspace \left(\tfrac{p+1}{2} \right), \ [P_{M_{p}}],
\end{array}
$$
where $A(r,\delta) = (r-1) \, \left( r (1+\delta r)(1+\delta) \right)^{-1/2} > 0$.
\medskip

\noindent Further, making use of the Stirling's approximation for the Gamma function, it follows
that
$$
\begin{array}{lll}
B_{p0}^{B}(\mathcal{B}_{p0}) & \approx & A(r,\delta) \, \pi^{1/2} \, 2^{(p-1)/2} \, p^{-(p+1)/2}
\left( \dfrac{r}{r-1} \right)^{(p(r-1))/2} \left( \dfrac{(1+\delta)^r}{1+\delta
r} \right)^{p/2} \smallskip \\ & & 2^{-(p-1)/2} \, p^{p/2} \, exp(-p/2) \\ \\
& = & A(r,\delta) \, \pi^{1/2} \, p^{-1/2} \left( \dfrac{r}{r-1} \right)^{(p(r-1))/2} \, \left(
\dfrac{(1+\delta)^r}{1+\delta r} \right)^{p/2} exp(-p/2), \ [P_{M_{p}}].
\smallskip
\end{array}
$$
So,
$$
\begin{array}{lll}
\lim_{n \rightarrow \infty} B_{p0}^{B}(\mathcal{B}_{p0}) & = & \lim_{p \rightarrow \infty} \left(
\negthinspace \left( \negthinspace \dfrac{r}{r-1} \negthinspace \right)^{\negthinspace r-1}
\frac{(1+\delta)^r}{1+\delta r} \, exp(-1) \negthinspace \right)^{\negthinspace p/2} \\ \\
& = & \left\{ \begin{array}{ll} 0, & \hbox{ if\, $h(r,\delta) \leq 1, \ [P_{M_{p}}]$,} \smallskip \\
\infty, & \hbox{ if\, $h(r,\delta) > 1, \ [P_{M_{p}}]$,} \end{array} \right. \medskip
\end{array}
$$
where $h(r,\delta) = \left( \dfrac{r}{r-1} \right)^{r-1} \dfrac{(1+\delta)^r}{1+\delta r} \,
exp(-1)$. \smallskip This proves v) and completes the proof of the theorem.
\end{enumerate}

\end{appendices}

\bibliography{Bibliography-Article}

\end{document}